\documentclass{article}
\usepackage{fullpage}
\usepackage[colorlinks,backref]{hyperref}
\usepackage{graphicx}
\usepackage{amsmath,amsthm,amssymb,enumerate}
\usepackage[normalem]{ulem} 
\usepackage{euscript,mathrsfs}
\usepackage[left=3cm,right=3cm,top=3.5cm,bottom=3.5cm]{geometry}
\usepackage{color}
\usepackage{bm}
\usepackage{bbm}
\catcode`\@=11 \@addtoreset{equation}{section}

\catcode`\@=12
\usepackage{color}
\allowdisplaybreaks

\newtheorem{Theorem}{Theorem}[section]
\newtheorem{Proposition}[Theorem]{Proposition}
\newtheorem{Lemma}[Theorem]{Lemma}
\newtheorem{Corollary}[Theorem]{Corollary}

\theoremstyle{definition}
\newtheorem{Definition}[Theorem]{Definition}

\newtheorem{Remark}[Theorem]{Remark}

\newcommand{\bTheorem}[1]{
\begin{Theorem} \label{T#1} }
\newcommand{\eT}{\end{Theorem}}

\newcommand{\bProposition}[1]{
\begin{Proposition} \label{P#1}}
\newcommand{\eP}{\end{Proposition}}

\newcommand{\bLemma}[1]{
\begin{Lemma} \label{L#1} }
\newcommand{\eL}{\end{Lemma}}

\newcommand{\bCorollary}[1]{
\begin{Corollary} \label{C#1} }
\newcommand{\eC}{\end{Corollary}}

\newcommand{\bRemark}[1]{
\begin{Remark} \label{R#1} }
\newcommand{\eR}{\end{Remark}}

\newcommand{\bDefinition}[1]{
\begin{Definition} \label{D#1} }
\newcommand{\eD}{\end{Definition}}

\newcommand{\dif}{\mathrm{d}}

\newcommand{\mf}{\mathfrak{F}}
\newcommand{\mr}{\mathbb{R}}
\newcommand{\prst}{\mathbb{P}}
\newcommand{\p}{\mathbb{P}}
\newcommand{\T}{\mathbb{T}}
\newcommand{\stred}{\mathbb{E}}

\newcommand{\mn}{\mathbb{N}}

\newcommand{\mt}{\mathbb{T}^3}

\newcommand{\bu}{\mathbf u}

\newcommand{\tor}{\mathbb{T}^3}

\newcommand{\StoB}{\left(\Omega, \mathfrak{F},(\mathfrak{F}_t )_{t \geq 0},  \mathbb{P}\right)}

\newcommand{\bfu}{\mathbf{u}}

\newcommand{\bfr}{\mathbf{r}}

\newcommand{\bFormula}[1]{
\begin{equation} \label{#1}}
\newcommand{\eF}{\end{equation}}

\newcommand{\vr}{\varrho}
\newcommand{\vre}{\vr_\ep}

\newcommand{\vue}{\vu_\ep}

\newcommand{\vu}{\vc{u}}
\newcommand{\vc}[1]{{\bf #1}}

\newcommand{\Div}{{\rm div}_x}
\newcommand{\Grad}{\nabla_x}

\newcommand{\tn}[1]{\mathbb{#1}}
\newcommand{\dx}{\,{\rm d} {x}}

\newcommand{\dt}{\,{\rm d} t }

\newcommand{\vm}{\vc{m}}

\newcommand{\D}{{\rm d}}
\newcommand{\ep}{\varepsilon}

\newcommand{\R}{\mathbb{R}}

\newcommand{\E}{\mathbb{E}}
\newcommand{\expe}[1]{ \mathbb{E} \left[ #1 \right] }
\newcommand{\intON}[1]{\int_{\mathbb{T}^3} #1 \dx }

\definecolor{Cgrey}{rgb}{0.85,0.85,0.85}
\definecolor{Cblue}{rgb}{0.50,0.85,0.85}
\definecolor{Cred}{rgb}{1,0,0}
\definecolor{fancy}{rgb}{0.10,0.85,0.10}

\newcommand{\Dif}{{\rm d}}
\newcommand{\DD}{D^d_t}
\newcommand{\DS}{\mathbb{D}^s_t}

\newcommand{\N}{\mathbb N}
\newcommand{\intTor}[1]{\int_{\tor} #1 \,dx}

\newcommand\Cbox[2]{%
    \newbox\contentbox%
    \newbox\bkgdbox%
    \setbox\contentbox\hbox to \hsize{%
        \vtop{
            \kern\columnsep
            \hbox to \hsize{%
                \kern\columnsep%
                \advance\hsize by -2\columnsep%
                \setlength{\textwidth}{\hsize}%
                \vbox{
                    \parskip=\baselineskip
                    \parindent=0bp
                    #2
                }%
                \kern\columnsep%
            }%
            \kern\columnsep%
        }%
    }%
    \setbox\bkgdbox\vbox{
        \color{#1}
        \hrule width  \wd\contentbox %
               height \ht\contentbox %
               depth  \dp\contentbox
        \color{black}
    }%
    \wd\bkgdbox=0bp%
    \vbox{\hbox to \hsize{\box\bkgdbox\box\contentbox}}%
    \vskip\baselineskip%
}

\DeclareMathOperator*{\esssup}{ess\,sup}

\DeclareMathOperator{\divv}{div}
\DeclareMathOperator{\diver}{div}

\date{}

\begin{document}


\title{Measure-valued solutions to the stochastic compressible Euler equations and incompressible limits}

\author{
Martina Hofmanov\'a \footnotemark[1] \and 
Ujjwal Koley \footnotemark[2]
\and Utsab Sarkar \footnotemark[3]
} 

\date{\today}

\maketitle

\centerline{$^*$ Fakult\"{a}t f\"{u}r Mathematik, Universit\"{a}t Bielefeld}
\centerline{D-33501 Bielefeld, Germany}
\centerline{hofmanova@math.uni-bielefeld.de}

\medskip
\centerline{$^\dagger$ Centre for Applicable Mathematics, Tata Institute of Fundamental Research}
\centerline{P.O. Box 6503, GKVK Post Office, Bangalore 560065, India}
\centerline{ujjwal@math.tifrbng.res.in}

\medskip
\centerline{$^\ddagger$ Centre for Applicable Mathematics, Tata Institute of Fundamental Research}
\centerline{P.O. Box 6503, GKVK Post Office, Bangalore 560065, India}\centerline{utsab@tifrbng.res.in}

\begin{abstract}
We introduce a new concept of dissipative measure-valued martingale solutions to the stochastic compressible Euler equations. These solutions are weak in the probabilistic sense i.e., the probability space and the driving Wiener process are an integral part of the solution. We derive the relative energy inequality for the stochastic compressible Euler equations and, as a corollary, we exhibit pathwise weak-strong uniqueness principle. Moreover, making use of the relative energy inequality, we investigate the low Mach limit (incompressible limit) of underlying system of equations. As a main novelty with respect to the related literature, our results apply to general nonlinear multiplicative stochastic perturbations of Nemytskij type.

\end{abstract}


\tableofcontents

\section{Introduction}
\label{P}
The concept of measure-valued solutions to deterministic partial differential equations (PDEs) was first introduced by DiPerna (\cite{DL}) in the context of conservation laws. 
The measure-valued solutions for compressible fluid dynamics equations has been introduced by Neustupa in \cite{Neus}, and subsequently revisited by Feireisl et. al. in \cite{Fei01}, where the authors have introduced the concept of so-called \emph{dissipative} measure-valued solutions. 
In this paper, we introduce a notion of dissipative measure-valued  solution for the stochastically forced system of the \emph{compressible barotropic
Euler system} describing the time evolution of the mass density $\vr$ and the bulk velocity $\vu$ of a fluid. The system of equations reads
\begin{align} \label{P1}
\D \vr + \Div (\vr \vu) \dt &= 0,\\ \label{P2}
\D (\vr \vu) + \left[ \Div (\vr \vu \otimes \vu)+  \Grad p(\vr) \right]  \dt  &=  \mathbb{G} (\vr, \vr \vu) \,\D W.
\end{align}
Here $p(\vr):=\vr^\gamma$ is the pressure; $\gamma>1$ denotes the adiabatic exponent, and to avoid the well-known difficulties related to the behaviour of fluid flows near the boundary, we study \eqref{P1}--\eqref{P2} on the $3$-dimensional torus $\mathbb{T}^3$.
The driving stochastic process $W$ is a cylindrical Wiener process defined on some probability space $(\Omega,\mathfrak{F},\p)$ and the coefficient $\mathbb{G}$ is generally nonlinear and satisfies suitable growth assumptions (see Section \ref{E} for the complete list of assumptions). The problem is closed by prescribing the random \emph{initial} data
\begin{equation} \label{P4}
\vr(0,\cdot) = \vr_0, \ \vr \vu(0, \cdot) = (\vr \vu)_0,
\end{equation}
with sufficient spatial regularity specified later. 

It is well-known that  smooth solutions to deterministic counterpart of \eqref{P1}--\eqref{P2} exist only for a finite lap of time after which singularities may develop for a generic class of initial data, and weak solutions must be sought. Moreover, weak solutions may not be uniquely determined by their initial data and an admissibility condition must be imposed to single out the physically correct solution. However, in light of the seminal work by De Lellis $\&$ Sz\'ekelyhidi \cite{DelSze3,DelSze}, and further confirmed by Chiodaroli et. al. \cite{chio01}, it turns out that the deterministic compressible Euler system is desperately ill-posed. Even if the initial data is smooth, the global existence and uniqueness of solutions can fail. Moreover, the question of existence of global-in-time weak solutions to \eqref{P1}--\eqref{P2} for \emph{general} initial data remains open. In the stochastic set-up, existence of a martingale solution for \eqref{P1}--\eqref{P2} was established in \cite{BV13} in \emph{one} space-dimension. A stochastic variant of deterministic results, as depicted in \cite{DelSze3,Feireisl002}, for the stochastic compressible Euler system have been achieved recently by Breit et. al. in \cite{BrFeHo2017}. In fact, it was shown in \cite{BrFeHo2017} that the same difficulty (ill-posedness) persists even in the presense of a random forcing. Incidentally, the existence of global strong solutions for \eqref{P1}--\eqref{P2} (i.e. the stopping time $\mathfrak t$ in Definition \ref{def:strsol} reaches $T$) is not expected.

In view of the above discussion, our goal is to present a suitable concept of \emph{dissipative} measure-valued solutions to \eqref{P1}--\eqref{P2} and discuss their basic properties. Broadly speaking, they are
measure-valued solutions of the stochastic compressible Euler system \eqref{P1}--\eqref{P2} satisfying an appropriate form of energy inequality and exist globally in time for any finite energy initial data.
The reason we consider measure-valued solutions is twofold: First, we see measure-valued solutions as possibly the largest class in which family of smooth (classical) solutions is stable. This is an undemanding consequence of weak (measure-valued)-strong uniqueness principle; second, solutions to an expansive class of approximate problems - emanating from numerical schemes - can be shown to converge to a measure-valued solution whereas convergence to a weak solution is either not known or requires strenuous efforts to prove. Measure-valued solution for fluid flows driven by \emph{multiplicative} noise is a relatively new area of pursuit, and to the best of our knowledge, this is the first attempt to characterize the concept of measure-valued solution for such equations. Having said this, we mention that for incompressible Euler equations driven by \emph{additive} noise, the existence of measure valued solution was first established by Kim in \cite{Kim2}. He has shown that the existence can be obtained from a sequence of solutions of stochastic Navier Stokes equations using tools from martingale selection. Recently, a concept of dissipative solutions to incompressible Euler equations driven by \emph{additive} noise has been introduced by Breit $\&$ Moyo in \cite{BrMo}.

To motivate our definition of measure-valued solutions, we need a sequence of approximate solutions to \eqref{P1}--\eqref{P2} together with an appropriate energy inequality. To give a concrete example, we consider an approximation by viscous compressible fluids subject to stochastic forcing described by the Navier--Stokes system
\begin{align} \label{P1NS}
\D \vr_{\varepsilon} + \Div (\vr_{\varepsilon} \vu_{\varepsilon}) \dt &= 0,\\
 \label{P2NS}
\D (\vr_{\varepsilon} \vu_{\varepsilon}) + \left[ \Div (\vr_{\varepsilon} \vu_{\varepsilon} \otimes \vu_{\varepsilon}) +  \Grad p(\vr_{\varepsilon}) \right] \dt  &= {\varepsilon}\,\Div \mathbb{S} (\Grad \vu_{\varepsilon}) \,dt + \mathbb{G} (\vr_{\varepsilon}, \vr_{\varepsilon} \vu_{\varepsilon}) \,\D W.
\end{align}
Here
$\mathbb{S}(\Grad \vu_{\varepsilon})$ is the viscous stress tensor for which we assume Newton's rheological law
\begin{equation*} 
\mathbb{S}(\Grad \vu_{\varepsilon}) = \delta \left( \Grad \vu_{\varepsilon} + \Grad^t \vu_{\varepsilon} - \frac{2}{3} \Div \vu_{\varepsilon} \mathbb{I} \right) + \lambda \,\Div \vu_{\varepsilon} \mathbb{I}, \qquad \delta > 0, \ \lambda \geq 0.
\end{equation*}
It is also well-known that smooth solutions to \eqref{P1NS}--\eqref{P2NS} satisfy the following energy inequality
\begin{align} \label{en_01}
\begin{aligned}
& \int_{\T^3} \bigg[ \frac{1}{2} \varrho_{\varepsilon} | {\bf u}_{\varepsilon} |^2  + \frac{\rho_{\varepsilon}^\gamma}{\gamma-1} \bigg] \dx 
+ {\varepsilon} \int_0^t \int_{\T^3}  \mathbb{S} (\nabla_x  {\bf u}_{\varepsilon}): \nabla_x  {\bf u}_{\varepsilon} \dx\,  {\rm d}s 
\leq \int_{\T^3} \bigg[ \frac{1}{2}\frac{|\varrho_{\varepsilon}\bfu_{\varepsilon}(0)|^2}{\varrho_{\varepsilon}(0)}  + \frac{\varrho_{\varepsilon}^\gamma(0)}{\gamma-1}  \bigg] \dx \\
& \qquad +\sum_{k=1}^\infty\int_0^t\bigg(\int_{\T^3}\mathbf{G}_k (\varrho_{\varepsilon}, \varrho_{\varepsilon} {\bf u}_{\varepsilon})\cdot{\bf u}_{\varepsilon}\dx\bigg){\rm d} W_k
+ \frac{1}{2}\sum_{k = 1}^{\infty}  \int_0^t
\int_{\T^3} \varrho_{\varepsilon}^{-1}| \mathbf{G}_k (\varrho_{\varepsilon}, \varrho_{\varepsilon} {\bf u}_{\varepsilon}) |^2 \, {\rm d}s
\end{aligned}
\end{align}

The study of the  system \eqref{P1NS}--\eqref{P2NS} was first initiated by Breit $\&$ Hofmanova in \cite{BrHo}, where the global-in-time existence of \emph{finite energy weak martingale solutions} is shown. These solutions are weak in the analytical sense (derivatives only exists in the sense of distributions) and weak in the probabilistic sense (the probability space is an integral part of the solution) as well. Moreover, a relative energy inequality, based on the pioneering work of Dafermos (\cite{Daf4}), has been recently established in Breit et. al. \cite{BrFeHo2015A}. As a corollary, they also demonstrated pathwise (and in law) weak-strong uniqueness property for the system \eqref{P1NS}--\eqref{P2NS}. The main deterrence against plausible extension of the deterministic argument, in deriving relative energy inequality \cite[Section 3]{BrFeHo2015A}, essentially lies in successful seizure of ``noise-noise'' interaction term (a perfect square term) between a smooth given strong solution and a martingale solution of \eqref{P1NS}--\eqref{P2NS}. Note that, in the realm of stochastic scalar conservation laws (refer to \cite{BhKoleyVa, BhKoleyVa_01, BisKoleyMaj, Koley1, Koley2}), ``noise-noise'' interaction term plays a pivotal role in the ``doubling of variables'' argument, \`{a} la Kru\v{z}kov. Furthermore, we mention that weak-strong property was used in \cite{BrFeHo2015A} to show convergence of weak solutions in the inviscid-incompressible limit.

In this paper, our interest is directed to the weak (measure-valued)--strong uniqueness principle for dissipative measure-valued solution. In what follows, we first introduce a notion of dissipative measure-valued solution for stochastic compressible Euler equations, keeping in mind that this framework will enable us to exhibit weak-strong uniquness property. Roughly speaking, the major stumbling block to emulate the ideas from \cite{BrFeHo2015A} lies in the successful identification (in the limit) of the martingale terms present in momentum equations \eqref{P2NS} and energy inequality \eqref{en_01}.
In general, due to lack of regular estimates for approximate solutions, it seems not possible to identify, in the limit, the martingale term in the momentum equation. In particular, it is neither possible to borrow results from Debussche et. al. \cite{Debussche} nor to use strong/weak continuity property of linear operators (like It\^o's integral) between two Banach spaces, to pass to the limit in the martingale terms. We can only conclude that the limit object is a martinagle. On the other hand, a closer look at the proof of weak-strong uniqueness in \cite{BrFeHo2015A} reveals that it is enough to have knowledge about the cross variation of a martingale solution with a strong solution. We are able to successfully leverage this observation further to establish the weak (measure-valued)--strong principle for \eqref{P1}--\eqref{P2}. Indeed, as a conceivable remedy, this is done by stipulating the correct cross variation with a smooth process in the definition of measure-valued solution (refer to Definition~\ref{def:dissMartin}).

Finally, to display the effectiveness of the relative energy inequality, we apply the same to find out the incompressible limit of stochastic compressible flows. Note that it is arguably rather difficult to show directly that the solutions of compressible Euler equations converge to the solutions of incompressible Euler equations in the low Mach number regime. But it is much easier, as it shown in Section~\ref{singular}, to show that the compressible Euler solutions converge to some strong kind of solutions to the incompressible Euler equations, thanks to weak-strong uniqueness property. 
In the realm of low Mach limit, there is a large literature related to incompressible limit for the deterministic counterpart of \eqref{P1}--\eqref{P2}, being the first important contribution by Kleinermann and Majda \cite{KM1}. For the so called \emph{well-prepared} data, justification of incompressible limit has been made rigorous by many authors (\cite{Asano1987, SCHO2, MESC1, Uka}). All these authors assume that the solutions of the compressible flow are smooth. However, as is well known, solutions of the compressible Euler system develop singularities in a finite time no matter how smooth and/or small the initial data are. The hypothesis of smoothness of solutions is therefore quite restrictive and even not appropriate in the context of compressible inviscid fluids. In light of the above discussions, to find out the incompressible limit, we follow the strategy of Feireisl et. al. \cite{Fei02} based on measure-valued solution. We also mention that this concept of dissipative measure valued solution and weak (measure-valued)--strong uniqueness principle are used recently in the analysis of convergence of certain numerical schemes \cite{K1}. Furthermore, a suitable concept of dissipative measure valued solution has been introduced for stochastic equations of incompressible fluid flow in \cite{K2}.

The rest of the paper is organized as follows:  we describe all mathematical/technical framework and state the main results in Section~\ref{E}. In Section~\ref{proof1}, we present a proof of existence of dissipative measure-valued martingale solutions using stochastic compactness. Section~\ref{proof2} is devoted on deriving the weak (measure-valued) -- strong uniqueness principle by making use of a suitable relative energy inequality. Finally, in Section~\ref{singular}, we explore another application of weak (measure-valued)--strong uniqueness property by justifying rigorously low Mach limit of the stochastic compressible Euler system \eqref{P1}--\eqref{P2}.

\section{Mathematical Framework and Main Results} 
\label{E}

In this section we present relevant mathematical tools to be used in the subsequent analysis and state main results of this paper. To begin, we fix an arbitrary large time horizon $T>0$. Throughout this paper, we use the letter $C$ to denote various generic constants. There are situations where constant may change from line to line, but the notation is kept unchanged so long as it does not impact the central idea. For a generic set $E$, let $\mathcal{M}_b(E)$ denotes the space of bounded Borel measures on $E$ whose norm is given by the total variation of measures. It is the dual space to the space of continuous functions vanishing at infinity $C_0(E)$ equipped with the supremum norm. Moreover, let $\mathcal{P}(E)$ be the space of probability measures on $E$.

\subsection{Analytic framework}
Let $H$ be a separable Hilbert space. Given $\alpha \in (0,1)$, let $W^{\alpha,2}(0,T;H)$ denotes a $H$-valued Sobolev space which is characterized by its norm
$$
\| u\|^2_{W^{\alpha,2}(0,T;H)}:= \int_0^T \| u(t)\|^2_{H}\,dt + \int_0^T \int_0^T \frac{\| u(t)-u(s)\|^2_H}{|t-s|^{1+2\alpha}}\,dt\,ds.
$$
We recall the following compact embedding result from Flandoli $\&$ Gatarek \cite[Theorem 2.2]{FlandoliGatarek}
\begin{Lemma}
\label{comp}
If $A \subset B$ are two Banach spaces with compact embedding, and real numbers $\alpha \in(0,1)$ satisfy $\alpha >1/2$, then the embedding
$$
W^{\alpha,2}(0,T;A) \subset C([0,T]; B)
$$
is compact.
\end{Lemma}

\subsubsection{Young measures, concentration defect measures}
\label{ym}
We shall briefly recapitulate the notion of Young measures and related results which have been used frequently in the text. For a nice overview of applications of
the Young measure theory to hyperbolic conservation laws, we refer to Balder \cite{Balder}. In what follows, let us first assume that $(X, \mathcal{M}, \mu)$ is a sigma finite measure space. A Young measure from $X$ into $\R^M$ is a weakly measurable function $\mathcal{V}: X \rightarrow \mathcal{P}(\R^M)$ in the sense that $x \mapsto \mathcal{V}_{x}(B)$ is $\mathcal{M}$-measurable for every Borel set $B$ in $\R^M$. In what follows, we make use of the following generalization of the classical result on Young measures, see \cite[Section~2.8]{BrFeHobook}.
\begin{Lemma}\label{lem001}
Let $N,M\in\mathbb{N}$, $\mathcal{O} \subset \R^N\times (0,T) $ and let $({\bf U}_n)_{n \in \N}$, ${\bf U}_n: \Omega\times\mathcal{O}  \to \R^M$,  be a sequence of random variables such that
$$
\E \big[ \|{\bf U}_n \|^p_{L^p(\mathcal{O})}\big] \le C, \,\, \text{for a certain}\,\, p\in(1,\infty).
$$
Then there exists a new subsequence $(\tilde {\bf U}_n)_{n \in \N}$ (not relabeled), defined on the standard probability space $\big([0,1], \overline{\mathcal{B}[0,1]}, \mathcal{L} \big)$, and a parametrized family ${\lbrace \mathcal{\tilde V}^{\omega}_{x} \rbrace}_{x \in \mathcal{O}}$ (superscript $\omega$ emphasises the dependence on $\omega$) of random probability measures on $\R^M$, regarded as a random variable taking values in $\big(L_{w^*}^{\infty}(\mathcal{O}; \mathcal{P}(\R^M)), w^* \big)$, such that ${\bf U}_n$ has the same law as $\tilde {\bf U}_n$, i.e.
$
{\bf U}_n \sim_{d} \tilde {\bf U}_n,
$
and  the following property holds: for any Carath\'eodory function $G=G(x, Z), x \in \mathcal{O}, Z \in \R^M$, such that
$$
|G(x,Z)| \le C(1 + |Z|^q), \quad 1 \le q < p, \,\,\text{uniformly in}\,\,x,
$$
implies $\mathcal{L}$-a.s.,
$$
G(\cdot, \tilde {\bf U}_n) \rightharpoonup \overline{G}\,\,\text{in}\,\, L^{p/q}(\mathcal{O}), 
\,\, \text{where}\,\, \overline{G}(x) = \langle \mathcal{\tilde V}^{\omega}_{(\cdot)}; G(x,\cdot)\rangle := \int_{\R^M} G(x, z)\,\D \mathcal{\tilde V}^{\omega}_{x}(z), \,\,\text{for a.a.}\,\, x \in \mathcal{O}.
$$
\end{Lemma}
\noindent It is well-known that Young measure theory works well when we try to extract limits of bounded continuous functions. Next, we examine what happens for those functions $H$ for which we only know that 
\begin{equation*}
\E \big[ \|H({\bf U}_n)\|^p_{L^1(\mathcal{O})}\big] \le C, \,\, \text{for a certain}\,\, p\in(1,\infty), \, \mbox{uniformly in } n.
\end{equation*}
In fact, using the fact that $L^1(\mathcal{O})$ is embedded in the space of bounded Radon measures $\mathcal{M}_b(\mathcal{O})$, we can infer that $\p$-a.s.
\begin{align*}
\mbox{weak-* limit in} \, \mathcal{M}_b(\mathcal{O})\,\, \mbox{of} \, \,H({\bf U}_n) = \langle \mathcal{\tilde V}^{\omega}_{x}; H\rangle \,dx+ H_{\infty},
\end{align*}
where $H_{\infty} \in \mathcal{M}_b(\mathcal{O})$, and $H_{\infty}$ is called \emph{concentration defect measure (or concentration Young measure)}. Note that, a simple truncation analysis and Fatou's lemma reveal that $\p$-a.s. $\| \langle \mathcal{\tilde V}^{\omega}_{(\cdot)}; H\rangle\|_{L^1(\mathcal{O})} \leq C$ and thus $\p$-a.s. $\langle \mathcal{\tilde V}^{\omega}_{x}; H \rangle $ is finite for a.e. $x\in \mathcal{O}$. In what follows, regarding the concentration defect measure, we shall make use of the following crucial lemma. For a proof of this lemma modulo cosmetic changes, we refer to Feireisl et. al. \cite[Lemma 2.1]{Fei01}. 


\begin{Lemma}
\label{lemma001}
Let $\{{\bf U}_n\}_{n > 0}$, ${\bf U}_n: \Omega \times \mathcal{O} \rightarrow \mathbb{R}^M$ be a sequence generating a Young measure $\{\mathcal{V}^{\omega}_y\}_{y\in \mathcal{O}}$, where $\mathcal{O}$ is a measurable set in $\mathbb{R}^N \times (0,T)$. Let $G: \mathbb{R}^M \rightarrow [0,\infty)$
be a continuous function such that
\begin{equation*}
\sup_{n >0} \E\big[\|G({\bf U}_n)\|^p_{L^1(\mathcal{O})} \big]< \infty, \, \text{for a certain}\,\, p\in(1,\infty),
\end{equation*}
and let $F$ be continuous such that
\begin{equation*}
F: \mathbb{R}^M \rightarrow \mathbb{R}, \quad |F(\bm{z})|\leq G(\bm{z}), \mbox{ for all } \bm{z}\in \mathbb{R}^M.
\end{equation*}
Let us denote $\p$-a.s.
\begin{equation*}
{F_{\infty}} := {\tilde{F}}- \langle \mathcal{\tilde V}^{\omega}_{y}, F(\textbf{v}) \rangle \,dy, \quad 
{G_{\infty}} := {\tilde{G}}- \langle \mathcal{\tilde V}^{\omega}_{y}, G(\textbf{v}) \rangle \,dy.
\end{equation*}
Here ${\tilde{F}}, {\tilde{G}} \in \mathcal{M}_b(\mathcal{O})$ are weak-$*$ limits of $\{F({\bf U}^n)\}_{n > 0}$, $\{G({\bf U}^n)\}_{n > 0}$ respectively in $\mathcal{M}_b(\mathcal{O})$. Then $\p$-a.s. $|F_{\infty}| \leq G_{\infty}$.
\end{Lemma}

\subsection{Stochastic framework}

Here we specify the definition of the stochastic forcing term.
Let $(\Omega,\mf,(\mf_t)_{t\geq0},\prst)$ be a stochastic basis with a complete, right-continuous filtration. The stochastic process $W$ is a cylindrical $(\mf_t)$-Wiener process in a separable Hilbert space $\mathfrak{U}$. It is formally given by the expansion
$$W(t)=\sum_{k\geq 1} e_k W_k(t),$$
where $\{ W_k \}_{k \geq 1}$ is a sequence of mutually independent real-valued Brownian motions relative to $(\mf_t)_{t\geq0}$ and $\{e_k\}_{k\geq 1}$ is an orthonormal basis of $\mathfrak{U}$.
To give the precise definition of the diffusion coefficient $\mathbb{G}$, consider $\varrho\in L^\gamma(\mt)$, $\varrho\geq0$, and $\bfu\in L^2(\mt)$ such that $\sqrt\varrho\bfu\in L^2(\mt)$. 
Denote $\bf m=\varrho\bf u$ and let $\,\mathbb{G}(\varrho,{\bf m}):\mathfrak{U}\rightarrow L^1(\mt)$ be defined as follows
$$\mathbb{G}(\varrho,{\bf m})e_k=\mathbf{G}_k(\cdot,\varrho(\cdot),{\bf m}(\cdot)).$$
The coefficients $\mathbf{G}_{k}:\mt\times\mr\times\mr^3\rightarrow\mr^3$ are $C^1$-functions that satisfy uniformly in $x\in\mt$
\begin{align}
|\vc{G}_{k}(\cdot,0,0)|+| \partial_\vr \vc{G}_k | + |\nabla_{{\bf m}} \vc{G}_k | &\leq \beta_k, \quad \sum_{k \geq 1} \beta_k  < \infty.
\label{FG2}
\end{align}
As discussed in \cite{BrHo}, we understand the stochastic integral as a process in the Hilbert space $W^{-m,2}(\mt)$, $m>3/2$. Indeed, it is easy to check that under the above assumptions on $\varrho$ and $\bf m$, the mapping $\mathbb{G}(\varrho,\varrho\bf m)$ belongs to $L_2(\mathfrak{U};W^{-m,2}(\mt))$, the space of Hilbert--Schmidt operators from $\mathfrak{U}$ to $W^{-m,2}(\mt)$.
Consequently, if\footnote{Here $\mathcal{P}$ denotes the predictable $\sigma$-algebra associated to $(\mf_t)$.}
\begin{align*}
\varrho&\in L^\gamma(\Omega\times(0,T),\mathcal{P},\dif\prst\otimes\dif t;L^\gamma(\mt)),\\
\sqrt\varrho\bfu&\in L^2(\Omega\times(0,T),\mathcal{P},\dif\prst\otimes\dif t;L^2(\mt)),
\end{align*}
and the mean value $(\varrho(t))_{\mt}$ is essentially bounded
then the stochastic integral
\[
\int_0^t \mathbb{G}(\vr, \vr \vu) \ {\rm d} W = \sum_{k \geq 1}\int_0^t \vc{G}_k (\cdot, \vr, \vr \vu) \ {\rm d} W_k
\]
 is a well-defined $(\mf_t)$-martingale taking values in $W^{-m,2}(\mt)$. Note that the continuity equation \eqref{P1} implies that the mean value $(\varrho(t))_{\mt}$ of the density $\varrho$ is constant in time (but in general depends on $\omega$).
Finally, we define the auxiliary space $\mathfrak{U}_0\supset\mathfrak{U}$ via
$$\mathfrak{U}_0:=\bigg\{u=\sum_{k\geq1}\beta_k e_k;\;\sum_{k\geq1}\frac{\beta_k^2}{k^2}<\infty\bigg\},$$
endowed with the norm
$$\|u\|^2_{\mathfrak{U}_0}=\sum_{k\geq1}\frac{\beta_k^2}{k^2},\quad v=\sum_{k\geq1}\beta_k e_k.$$
Note that the embedding $\mathfrak{U}\hookrightarrow\mathfrak{U}_0$ is Hilbert--Schmidt. Moreover, trajectories of $W$ are $\prst$-a.s. in $C([0,T];\mathfrak{U}_0)$.

Finally, we recall the ``Kolmogorov test''  for the existence of continuous modifications of real-valued stochastic processes.

\begin{Lemma}
	\label{lemma01}
	Let $Y={\lbrace Y(t) \rbrace}_{t \in [0,T]} $ be a real-valued stochastic process defined on a probability space $(\Omega,\mf,(\mf_t)_{t\geq0},\prst)$. Suppose that there are constants $\alpha >1, \beta >0$, and $C>0$ such that for all $s,t \in [0,T]$,
	\begin{align*}
	\E[|Y(t)-Y(s)|^\alpha] \le C |t-s|^{1 + \beta}.
	\end{align*}
	Then there exists a continuous modification of $Y$ and the paths of $Y$ are $\gamma$-H\"{o}lder continuous for every $\gamma \in [0, \frac{\beta}{\alpha})$.
\end{Lemma}

\subsection{Stochastic compressible Euler equations}

Since the weak (measure-valued)--strong uniqueness principle for dissipative measure-valued solutions require the existence of strong solution, we first recall the notion of local strong pathwise solution for stochastic compressible Euler equations, being first introduced in \cite{BrMe}. Such a solution is strong in both the probabilistic and PDE sense, at least locally in time. To be more precise, system \eqref{P1}--\eqref{P2} will be satisfied pointwise (not only in the sense of distributions) on the given stochastic basis associated to the cylindrical Wiener process $W$.

\begin{Definition}[Local strong pathwise solution] \label{def:strsol}

Let $\StoB$ be a stochastic basis with a complete right-continuous filtration. Let ${W}$ be an $(\mathfrak{F}_t) $-cylindrical Wiener process and $(\varrho_0,\bfu_0)$ be a $W^{s,2}(\T^3)\times W^{s,2}(\T^3)$-valued $\mathfrak{F}_0$-measurable random variable, for some $s>7/2$.
A triplet
$(\varrho,\vu,\mathfrak{t})$ is called a local strong pathwise solution to the system \eqref{P1}--\eqref{P4} provided
\begin{enumerate}[(a)]
\item $\mathfrak{t}$ is an a.s. strictly positive  $(\mathfrak{F}_t)$-stopping time;
\item the density $\varrho$ is a $W^{s,2}(\mt)$-valued $(\mathfrak{F}_t)$-progressively measurable process satisfying
$$\varrho(\cdot\wedge \mathfrak{t})  > 0,\ \varrho(\cdot\wedge \mathfrak{t}) \in C([0,T]; W^{s,2}(\mt)) \quad \mathbb{P}\text{-a.s.};$$
\item the velocity $\vu$ is a $W^{s,2}(\mt)$-valued $(\mathfrak{F}_t)$-progressively measurable process satisfying
$$ \vu(\cdot\wedge \mathfrak{t}) \in   C([0,T]; W^{s,2}(\mt))\quad \mathbb{P}\text{-a.s.};$$
\item  there holds $\prst$-a.s.
\[
\begin{split}
\varrho (t\wedge \mathfrak{t}) &= \varrho_0 -  \int_0^{t \wedge \mathfrak{t}} \Div(\varrho\vu ) \ \dif s, \\
(\varrho \vu) (t \wedge \mathfrak{t})  &= \varrho_0 \vu_0 - \int_0^{t \wedge \mathfrak{t}} \Div (\varrho\vu \otimes\vu ) \ \dif s 
- \int_0^{t \wedge \mathfrak{t}}\Grad \varrho^\gamma \ \dif s + \int_0^{t \wedge \mathfrak{t}} {\tn{G}}(\varrho,\varrho\vu ) \ \D W,
\end{split}
\]
for all $t\in[0,T]$.
\end{enumerate}
\end{Definition}

Note that classical solutions require spatial derivatives of $\vr$ and $\vu$ to be continuous $\prst$-a.s. This motivates the following definition.

\begin{Definition}[Maximal strong pathwise solution]\label{def:maxsol}
Fix a stochastic basis with a cylindrical Wiener process and an initial condition as in Definition \ref{def:strsol}. A quadruplet $$(\varrho,\vu,(\mathfrak{t}_R)_{R\in\mn},\mathfrak{t})$$ is a maximal strong pathwise solution to system \eqref{P1}--\eqref{P4} provided

\begin{enumerate}[(a)]
\item $\mathfrak{t}$ is an a.s. strictly positive $(\mathfrak{F}_t)$-stopping time;
\item $(\mathfrak{t}_R)_{R\in\mn}$ is an increasing sequence of $(\mathfrak{F}_t)$-stopping times such that
$\mathfrak{t}_R<\mathfrak{t}$ on the set $[\mathfrak{t}<T]$,
$\lim_{R\to\infty}\mathfrak{t}_R=\mathfrak t$ a.s. and
\begin{equation}\label{eq:blowup}
\sup_{t\in[0,\mathfrak{t}_R]}\|\vu(t)\|_{1,\infty}\geq R\quad \text{on}\quad [\mathfrak{t}<T] ;
\end{equation}
\item each triplet $(\varrho,\vu,\mathfrak{t}_R)$, $R\in\mn$,  is a local strong pathwise solution in the sense  of Definition \ref{def:strsol}.
\end{enumerate}
\end{Definition}

Regarding existence of a maximal pathwise solution, there are number of works available in literature in the context of various SPDE or SDE models, see for instance \cite{BMS,Elw}. For the specific result on compressible Euler equations, we refer to the work by Breit $\&$ Mensah in \cite[Theorem 2.4]{BrMe}.

\begin{Theorem}\label{thm:main}
Let $s>7/2$ and $(\varrho_0,\bfu_0)$ be an $\mathfrak{F}_0$-measurable, $W^{s,2}(\mt)\times W^{s,2}(\mt)$-valued random variable such that $\varrho_0>0$ $\p$-a.s. Then, under certain assumptions imposed on the coefficients $\mathbf{G}_k$,
there exists a unique maximal strong pathwise solution, in the sense of Definition \ref{def:maxsol}, $(\varrho,\vu,(\mathfrak{t}_R)_{R\in\mn},\mathfrak{t})$ to problem \eqref{P1}--\eqref{P4}
 with the initial condition $(\varrho_0,\vu_0)$.
\end{Theorem}

\subsection{Stochastic compressible Navier--Stokes equations}

There is a relatively small literature related to stochastic compressible Navier--Stokes equations, with the first contribution by Breit $\&$ Hofmanova in \cite{BrHo}. In this connection we also mention \cite{BrFeHo2015A} and \cite{BrFeHobook}, which improves the results in \cite{BrHo}. All the papers complement \eqref{P1NS}--\eqref{P2NS} with periodic boundary conditions. Here we recall the concept of finite energy weak martingale solutions to \eqref{P1NS}--\eqref{P2NS}. These solutions are weak in the analytical sense (derivatives only exists in the sense of distributions) and weak in the probabilistic sense (the probability space is an integral part of the solution) as well. Moreover, the time-evolution of the energy can be controlled in terms of its initial state. They exists globally in time.

\begin{Definition}[Finite energy weak martingale solution]
\label{def:weakMartin}
Let $\Lambda_{\varepsilon}$ be a Borel probability measure on $L^\gamma(\mathbb{T}^3)\times L^{2\gamma/\gamma+1}(\mathbb{T}^3)$. Then $\big[ \big(\Omega_{\varepsilon},\mathfrak{F}_{\varepsilon}, (\mathfrak{F}_{{\varepsilon},t})_{t\geq0},\mathbb{P}_{\varepsilon} \big);\varrho_{\varepsilon},\mathbf{u}_{\varepsilon}, W_{\varepsilon} \big]$ is a finite energy weak martingale solution of \eqref{P1NS}--\eqref{P2NS} if
\begin{enumerate}[(a)]
\item $ \big(\Omega_{\varepsilon},\mathfrak{F}_{\varepsilon}, (\mathfrak{F}_{{\varepsilon},t})_{t\geq0},\mathbb{P}_{\varepsilon} \big)$ is a stochastic basis with a complete right-continuous filtration,
\item $W_{\varepsilon}$ is a $(\mathfrak{F}_{{\varepsilon},t})$-cylindrical Wiener process,
\item the density $\varrho_{\varepsilon}$ satisfies $\varrho_{\varepsilon}\geq 0$, $t\mapsto \langle \varrho_{\varepsilon}(t, \cdot),\phi\rangle\in C[0,T]$ for any $\phi\in C^\infty(\mathbb{T}^3)$ $\mathbb{P}_{\varepsilon}$-a.s., the function $t\mapsto \langle \varrho_{\varepsilon}(t, \cdot),\phi\rangle$ is progressively measurable and 
\begin{align*}
\mathbb{E}_{\varepsilon}\, \bigg[ \sup_{t\in(0,T)}\Vert  \varrho_{\varepsilon}(t,\cdot)\Vert_{L^\gamma(\mt)}^p\bigg]<\infty 
\end{align*}
for all $1\leq p<\infty$,
\item the velocity field $\mathbf{u}_{\varepsilon}$ is $(\mathfrak{F}_{{\varepsilon},t})$-adapted and 
\begin{align*}
\mathbb{E}_{\varepsilon}\, \bigg[ \int_0^T\Vert  \mathbf{u}_{\varepsilon}\Vert^2_{W^{1,2}(\mt)}\,\mathrm{d}t\bigg]^p<\infty 
\end{align*}
for all $1\leq p<\infty$,
\item the momentum $\varrho_{\varepsilon}\mathbf{u}_{\varepsilon}$ satisfies  $t\mapsto \langle \varrho_{\varepsilon}\mathbf{u}_{\varepsilon},\bm\varphi\rangle\in C[0,T]$ for any $\bm{\varphi}\in C^\infty(\mathbb{T}^3)$ $\mathbb{P}$-a.s., the function $t\mapsto \langle \varrho_{\varepsilon}\mathbf{u}_{\varepsilon},\bm{\phi}\rangle$ is progressively measurable and 
\begin{align*}
\mathbb{E}_{\varepsilon}\, \bigg[ \sup_{t\in(0,T)}\Vert  \varrho_{\varepsilon}\mathbf{u}_{\varepsilon}\Vert_{L^\frac{2\gamma}{\gamma+1}(\mt)}^p\bigg]<\infty 
\end{align*}
for all $1\leq p<\infty$,
\item $\Lambda_{\varepsilon}=\mathbb{P}_{\varepsilon}\circ \big[ \varrho_{\varepsilon}(0),\varrho_{\varepsilon}\mathbf{u}_{\varepsilon}(0) \big]^{-1}$,
\item for all $\phi\in C^\infty(\mathbb{T}^3)$ and $\bm{\varphi}\in C^\infty(\mathbb{T}^3)$ we have
\begin{equation}\label{eq:energy}
\begin{aligned}
\langle \varrho_{\varepsilon}(t), \phi\rangle &= \langle\varrho_{\varepsilon}(0) , \phi\rangle - \int_0^{t}\langle\varrho_{\varepsilon}\mathbf{u}_{\varepsilon}, \nabla_x \phi\rangle\mathrm{d}s 
\\
\langle \varrho_{\varepsilon}\mathbf{u}_{\varepsilon}(t), \bm\varphi\rangle &= \langle \varrho_{\varepsilon}\mathbf{u}_{\varepsilon} (0), \bm{\varphi}\rangle - \int_0^{t}\langle \varrho_{\varepsilon}\mathbf{u}_{\varepsilon}\otimes\mathbf{u}_{\varepsilon}, \nabla_x  \bm{\varphi}\rangle\,\mathrm{d}s
+ {\varepsilon}\,\int_0^{t} \langle\mathbb{S}(\nabla_x \mathbf{u}_{\varepsilon})\,, \nabla_x  \bm{\varphi}\rangle\mathrm{d}s 
\\&
- \int_0^{t}\langle  p(\varrho_{\varepsilon}), \mathrm{div}_x\,\bm{\varphi}\rangle\,\mathrm{d}s
+\int_0^{t}\langle\mathbb{G}(\varrho_{\varepsilon},\varrho_{\varepsilon}\mathbf{u}_{\varepsilon}), \bm{\varphi}\rangle\,\mathrm{d}W
\end{aligned}
\end{equation}
$\mathbb{P}$-a.s. for all $t\in[0,T]$,
\item the energy inequality\index{energy inequality}
\begin{align}\label{EI2''}
\begin{aligned}
&
-\int_0^T \partial_t \psi \int_{\T^3} \bigg[ \frac{1}{2} \varrho_{\varepsilon} | {\bf u}_{\varepsilon} |^2  + \frac{ \varrho_{\varepsilon}^\gamma}{\gamma-1} \bigg] \,dx \,dt
+ {\varepsilon}\, \int_0^T \psi \int_{\T^3}  \mathbb{S} (\nabla_x  {\bf u}_{\varepsilon}): \nabla_x  {\bf u}_{\varepsilon} \dx\,  \dt\\
& \qquad \leq \psi(0)\int_{\T^3} \bigg[ \frac{1}{2}\frac{|\varrho_{\varepsilon}\bfu_{\varepsilon}(0)|^2}{\varrho_{\varepsilon}(0)}  + \frac{\varrho_{\varepsilon}^\gamma(0)}{\gamma-1}  \bigg] \dx
+\sum_{k=1}^\infty\int_0^T \psi \bigg(\int_{\T^3}\mathbf{G}_k (\varrho_{\varepsilon}, \varrho_{\varepsilon} {\bf u}_{\varepsilon})\cdot{\bf u}_{\varepsilon}\dx\bigg){\rm d} W_k\\
& \qquad \qquad + \frac{1}{2}\sum_{k = 1}^{\infty}  \int_0^T \psi
\int_{\T^3} \varrho_{\varepsilon}^{-1}| \mathbf{G}_k (\varrho_{\varepsilon}, \varrho_{\varepsilon} {\bf u}_{\varepsilon}) |^2 \,dx \,dt
\end{aligned}
\end{align}
holds $\mathbb P$-a.s., for all $\psi \in C_c^{\infty}([0,T))$, $\psi \ge 0$.
\end{enumerate}
\end{Definition}

\begin{Remark}
Note that we may take, for every $\varepsilon$, $\Omega_{\varepsilon} =[0,1]$, with $\mathfrak{F}_{\varepsilon}$ the $\sigma$-algebra of the Borelians on $[0,1]$, and $\mathbb{P}_{\varepsilon}$ the Lebesgue measure on $[0,1]$, see Skorohod \cite{sko}. Moreover, we can assume without loss of generality that there exists one common Wiener space $W$ for all $\varepsilon$. Indeed, this can be achieved by performing the compactness argument from any chosen subsequence ${\lbrace \varepsilon_n \rbrace}_{n \in \N}$ at once. However, we stress the fact that it may not be possible to preclude the dependence of $\varepsilon$ in the filtration $ (\mathfrak{F}_{{\varepsilon},t})_{t\geq0}$, due to lack of pathwise uniqueness for the underlying system.
\end{Remark}
The following existence theorem is shown in \cite[Theorem 1]{BrHo} (see also \cite[Chapter 4]{BrFeHobook} for a more general version).
\begin{Theorem}
\label{thm:Romeo1}
Let $\gamma>\frac{3}{2}$ and assume that $\Lambda_{\varepsilon}$ is a Borel probability measure on $L^\gamma(\mathbb{T}^3)\times L^{2\gamma/\gamma+1}(\mathbb{T}^3)$ such that
\begin{align*}
\Lambda_{\varepsilon}\Bigg\{(\varrho,\mathbf{m})\in L^\gamma(\mathbb{T}^3)\times L^{2\gamma/\gamma+1}(\mathbb{T}^3) \, :\, 0<M_1\leq \varrho\leq M_2,\, \mathbf{m}\vert_{\{\varrho=0\}}=0   \Bigg\}=1  
\end{align*}
with constants $0<M_1<M_2$. Furthermore, assume that the following moment estimate
\begin{align*}
\int_{L^\gamma_x\times  L^{2\gamma/\gamma+1}_x} \bigg\Vert\frac{1}{2}\frac{\vert  \mathbf{m} \vert^2}{\varrho}+ \frac{ \varrho^\gamma}{\gamma-1}  \bigg\Vert^p_{L^1(\T^3)}\, \mathrm{d}\Lambda_{\varepsilon}(\varrho,\mathbf{m})<\infty,
\end{align*}
holds for all $1\leq p<\infty$. Finally, assume that \eqref{FG2} holds. Then there exists a finite energy weak martingale solution of \eqref{P1NS}--\eqref{P2NS} in the sense of Definition \ref{def:weakMartin} with initial law $\Lambda_{\varepsilon}$.
\end{Theorem}

\subsection{Measure-valued solutions}
To introduce the notion of measure-valued solution, it is more convenient to reformulate the problem \eqref{P1}--\eqref{P2} in the \emph{conservative} variables $\varrho$ and ${\bf m}= \varrho {\bf u}$:
\begin{align} \label{P11}
\D \vr + \Div {\bf m} \dt &= 0,\\ \label{P21}
\D {\bf m} + \left[ \Div \bigg(\frac{{\bf m} \otimes {\bf m}}{\varrho}\bigg)+  \Grad p(\vr) \right]  \dt  &=  \mathbb{G} (\vr, {\bf m}) \,\D W.
\end{align}
Note that, in general, the energy inequality \eqref{EI2''} seems to be the only source of {\it a priori} bounds. However, it is well known that uniformly bounded sequence in $L^1(\T^3)$, in general, does not imply weak convergence of it due to the presence of oscillations and concentration effects. In fact, to accommodate the above mentioned singularities two kinds of tools are used:
\begin{itemize}
\item the Young measures, which are probability measures on the phase space and account for the persistence of oscillations in the solution;
\item concentration defect measures, which are measures on physical space-time, account for blow up type collapse due to possible concentration points.
\end{itemize}


\subsubsection{Dissipative measure-valued martingale solutions }
	
Motivated by the previous discussion, we are ready to introduce the concept of \textit{dissipative measure--valued martingale solution} to the stochastic compressible Euler system. In what follows, let 
\[
\mathcal{S} = \left\{ [\vr, \vm] \ \Big| \ \vr \geq 0, \ \vm \in \R^3 \right\}
\]
be the phase space associated to the Euler system. 

\begin{Definition}[Dissipative measure-valued martingale solution]
\label{def:dissMartin}
Let $\Lambda$ be a Borel probability measure on $L^\gamma(\mathbb{T}^3)\times L^{\frac{2\gamma}{\gamma+1}}(\mathbb{T}^3)$. Then $\big[ \big(\Omega,\mathfrak{F}, (\mathfrak{F}_{t})_{t\geq0},\mathbb{P} \big); \mathcal{V}^{\omega}_{t,x}, W \big]$ is a dissipative measure-valued martingale solution of \eqref{P11}--\eqref{P21}, with initial condition $\mathcal{V}^{\omega}_{0,x}$; if
\begin{enumerate}[(a)]
\item $\mathcal{V}^{\omega}$ is a random variable taking values in the space of Young measures on $L^{\infty}_{w^*}\big([0,T] \times \T^3; \mathcal{P}\big(\mathcal{S})\big)$. In other words, $\p$-a.s.
$\mathcal{V}^{\omega}_{t,x}: (t,x) \in [0,T] \times \T^3  \rightarrow \mathcal{P}(\mathcal{S})$ is a parametrized family of probability measures on $\mathcal{S}$,
\item $ \big(\Omega,\mathfrak{F}, (\mathfrak{F}_{t})_{t\geq0},\mathbb{P} \big)$ is a stochastic basis with a complete right-continuous filtration,
\item $W$ is a $(\mathfrak{F}_{t})$-cylindrical Wiener process,
\item the average density $\langle \mathcal{V}^{\omega}_{t,x}; \varrho \rangle$ satisfies $t\mapsto \langle \langle \mathcal{V}^{\omega}_{t,x}; \varrho \rangle(t, \cdot),\phi\rangle\in C[0,T]$ for any $\phi\in C^\infty(\mathbb{T}^3)$ $\mathbb{P}$-a.s., the function $t\mapsto \langle \langle \mathcal{V}^{\omega}_{t,x}; \varrho \rangle(t, \cdot),\phi\rangle$ is progressively measurable and 
\begin{align*}
\mathbb{E}\, \bigg[ \sup_{t\in(0,T)}\Vert  \langle \mathcal{V}^{\omega}_{t,x}; \varrho \rangle(t,\cdot)\Vert_{L^\gamma(\mt)}^p\bigg]<\infty 
\end{align*}
for all $1\leq p<\infty$,
\item the average momentum $\langle \mathcal{V}^{\omega}_{t,x}; \textbf{m} \rangle$ satisfies  $t\mapsto \langle \langle \mathcal{V}^{\omega}_{t,x}; \textbf{m} \rangle (t, \cdot),\bm\varphi\rangle\in C[0,T]$ for any $\bm{\varphi}\in C^\infty(\mathbb{T}^3)$ $\mathbb{P}$-a.s., the function $t\mapsto \langle \langle \mathcal{V}^{\omega}_{t,x}; \textbf{m} \rangle (t, \cdot),\bm{\phi}\rangle$ is progressively measurable and 
\begin{align*}
\mathbb{E}\, \bigg[ \sup_{t\in(0,T)}\Vert  \langle \mathcal{V}^{\omega}_{t,x}; \textbf{m} \rangle (t, \cdot) \Vert_{L^\frac{2\gamma}{\gamma+1}(\mt)}^p\bigg]<\infty 
\end{align*}
for all $1\leq p<\infty$,
\item $\Lambda=\mathcal{L}[\mathcal{V}^{\omega}_{0,x}]$,
\item the integral identity
\begin{equation} 
\begin{aligned}\label{first condition measure-valued solution}
&\int_{\T^3} \langle \mathcal{V}^{\omega}_{\tau,x}; \varrho \rangle \, \varphi (\tau, \cdot) \,dx -\int_{\T^3} \langle \mathcal{V}^{\omega}_{0,x}; \varrho \rangle\, \varphi (0,\cdot)\,dx 
= \int_{0}^{\tau} \int_{\T^3}  \langle \mathcal{V}^{\omega}_{t,x}; \textbf{m}\rangle \cdot \nabla_x \varphi \,dx \,dt,
\end{aligned}
\end{equation}
holds $\p$-a.s., for all $\tau \in[0,T)$, and for all $\varphi \in C^{\infty}(\T^3)$,
\item there exists a $W^{-m,2}(\T^3)$-valued square integrable continuous martingale $M^1_{E}$, such that the integral identity 
\begin{equation} \label{second condition measure-valued solution}
\begin{aligned}
&\int_{\T^3} \langle \mathcal{V}^{\omega}_{\tau,x};\textbf{m}\rangle \cdot \bm{\varphi}(\tau, \cdot) dx - \int_{\T^3} \langle \mathcal{V}^{\omega}_{0,x};\textbf{m}\rangle \cdot \bm{\varphi}(0,\cdot) dx \\
&\qquad = \int_{0}^{\tau} \int_{\T^3} \left[\left\langle \mathcal{V}^{\omega}_{t,x}; \frac{\textbf{m}\otimes \textbf{m}}{\varrho} \right\rangle: \nabla_x \bm{\varphi} + \langle \mathcal{V}^{\omega}_{t,x};p(\varrho)\rangle \divv_x \bm{\varphi} \right] dx dt \\
&\qquad \qquad+ \int_0^{\tau}\int_{\T^3} \bm{\varphi} \,dx\, dM^1_{E}(t) + \int_{0}^{\tau} \int_{\T^3}  \nabla_x \bm{\varphi}: d\mu_m,
\end{aligned}
\end{equation}
holds $\p$-a.s., for all $\tau \in [0,T)$, and for all $\bm{\varphi} \in C^{\infty}(\T^3;\mathbb{R}^3)$, where $\mu_m\in L^{\infty}_{w^*}\big([0,T]; \mathcal{M}_b({\T^3} )\big)$, $\p$-a.s., is a tensor--valued measure,
\item there exists a real-valued square integrable continuous martingale $M^2_{E}$, such that the following energy inequality
\begin{align} \label{third condition measure-valued solution}
\begin{aligned}
\mathcal{E}(t+) & \leq  \mathcal{E}(s-) \\
&\quad + \frac{1}{2} \int_s^{t} \bigg(\int_{\T^3} \sum_{k = 1}^{\infty} \left\langle \mathcal{V}^{\omega}_{\tau,x};\varrho^{-1}| \mathbf{G}_k (\varrho, \textbf{m}) |^2 \right\rangle \bigg)\, d\tau + \frac12\int_s^{t} \int_{\T^3} d \mu_e + \int_s^{t}  dM^2_{E} 
\end{aligned}			
\end{align}
holds $\p$-a.s., for all $0 \le s <t$ in $(0,T)$ with
\begin{align*}
&\mathcal{E}(t-):=\lim_{\tau \rightarrow 0+} \frac{1}{\tau} \int_{t -\tau}^t \Bigg(\int_{\T^3} \left\langle \mathcal{V}^{\omega}_{s,x}; \frac{1}{2} \frac{|\textbf{m}|^2}{\varrho} +P(\varrho) \right \rangle \,dx + \mathcal{D}(s)\Bigg) \,ds\\
&\mathcal{E}(t+):=\lim_{\tau \rightarrow 0+} \frac{1}{\tau} \int_t^{t +\tau} \Bigg(\int_{\T^3} \left\langle \mathcal{V}^{\omega}_{s,x}; \frac{1}{2} \frac{|\textbf{m}|^2}{\varrho} +P(\varrho) \right \rangle \,dx + \mathcal{D}(s)\Bigg) \,ds 
\end{align*}
Here $\mu_e\in L^{\infty}_{w^*}\big([0,T]; \mathcal{M}_b({\T^3} )\big)$, $\p$-almost surely,  $\mathcal{D}\geq 0$, $\p$-almost surely, and $\E \big[ \esssup_{t \in (0,T)}\mathcal{D}(t)\big] < \infty$, with initial energy
$$
\mathcal{E}(0-)= 
\int_{\mathbb{T}^3} \bigg(\frac{1}{2} \frac{|\textbf{m}_0|^2}{\varrho_0} +P(\varrho_0) \bigg)\,\dx.
$$
\item there exists a constant $C>0$ such that
\begin{equation} \label{fourth condition measure-valued solutions}
 \int_{0}^{\tau} \int_{\T^3} d|\mu_m| + \int_{0}^{\tau} \int_{\T^3} d|\mu_e| \leq C \int_{0}^{\tau} \mathcal{D}(t) \,\dt,
\end{equation}	
holds $\p$-a.s., for every $\tau \in (0,T)$. 
\item Given any random stochastic process $f$ adapted to $(\mathfrak{F}_{t})_{t\geq0}$ given by
$$
\mathrm{d}f  = D^df\,\mathrm{d}t  + \mathbb{D}^s f\,\mathrm{d}W,
$$
satisfying
\begin{equation*} 
f \in C([0,T]; W^{1,q} \cap C (\tor)), \quad 
\E\bigg[ \sup_{t \in [0,T]} \| f \|_{W^{1,q}}^2 \bigg]^q < \infty, \quad \text{$\mathbb{P}$-a.s. for all }\ 1 \leq q < \infty,
\end{equation*}
and
\begin{align*}
&D^d f \in L^q(\T^3;L^q(0,T;W^{1,q}(\mt))),\quad
\mathbb D^s f \in L^2(\T^3;L^2(0,T;L_2(\mathfrak U;L^2(\tor)))),\nonumber\\
&\bigg(\sum_{k\geq 1}|\mathbb D^s f(e_k)|^q\bigg)^\frac{1}{q} \in L^q(\T^3;L^q(0,T;L^{q}(\mt))),
\end{align*}
the cross variation between $f$ and the square integrable continuous martingale $M^1_E$ is given by 
\begin{align*}
\Big<\hspace{-0.14cm}\Big<f(t),  M^1_{E}(t)  \Big>\hspace{-0.14cm}\Big> 
= \sum_{i,j} \Bigg(\sum_{k = 1}^{\infty}  \int_0^t  \left \langle \left\langle \mathcal{ V}^{\omega}_{s,x}; \mathbf{G}_k (\varrho, {\bf m})\right\rangle, g_i \right \rangle \, \left \langle \mathbb{D}^s f(e_k), g_j \right\rangle  \,ds \Bigg) g_i \otimes g_j,
\end{align*}
where $g_j$'s are orthonormal basis for the space $W^{-r,2}(\T^N)$, for some $r>0$, and the bracket $\big<\cdot, \cdot \big>$ denotes inner product in the same space.
\end{enumerate}
\end{Definition}	

\begin{Remark}
Since our interest is directed to the weak (measure-valued)-strong uniqueness principle for dissipative measure-valued solution, the regularity assumptions in the item (k) on the process $f$ are indeed taylor made to successfully demonstrate the same principle for stochastic Euler equations, by making use of the relative energy inequality.
\end{Remark}

\begin{Remark}
Note that, in view of a standard Lebesgue point argument applied to \eqref{third condition measure-valued solution}, energy inequality holds for a.e. $0 \le s <t$ in $(0,T)$: 
\begin{equation}
\begin{aligned}
\label{energy_001}
	&  \int_{\T^3} \left\langle \mathcal{V}^{\omega}_{t,x}; \frac{1}{2} \frac{|\textbf{m}|^2}{\varrho} +P(\varrho) \right \rangle \,dx + \mathcal{D}(t)\\
	&\qquad \leq  \int_{\T^3}  \left\langle \mathcal{V}^{\omega}_{s,x}; \frac{1}{2} \frac{|\textbf{m}|^2}{\varrho} +P(\varrho)\right \rangle  \,dx + \mathcal{D}(s) + \frac{1}{2} \int_s^{t} \bigg(\int_{\T^3} \sum_{k = 1}^{\infty} \left\langle \mathcal{V}^{\omega}_{\tau,x};\varrho^{-1}| \mathbf{G}_k (\varrho, \textbf{m}) |^2 \right\rangle \bigg)\, d\tau \\
	& \qquad \qquad+ \frac12\int_s^{t} \int_{\T^3} d \mu_e + \int_s^{t}  dM^2_{E}, \,\,\p-a.s.
\end{aligned}
\end{equation}
However, for technical reasons, as pointed out in Section~\ref{proof2}, we require energy inequality to hold for \emph{all} $s, t \in (0,T)$. This can be achieved following the argument depicted in Section~\ref{proof1}.
\end{Remark}

\subsection{Statements of main results}

We now state main results of this paper. To begin with, regarding the existence of dissipative measure-valued martingale solutions, we have the following theorem.

\begin{Theorem} \label{thm:exist}
Suppose that $\gamma >\frac 32$ and $\vc{G}_k$ be continuously differentiable satisfying \eqref{FG2}. Let $(\varrho_{\varepsilon}, {\bf m}_{\varepsilon}= \varrho_{\varepsilon}\,{\bf u}_{\varepsilon})$ be a family of finite energy weak martingale solutions to the stochastic compressible Navier--Stokes system \eqref{P1NS}--\eqref{P2NS}. Let the corresponding initial data $\varrho_0, {\bf m}_{0}$, the initial law $\Lambda$, given on the space $L^\gamma (\T^3) \times L^{\frac{2 \gamma}{\gamma + 1}}(\T^3;R^3)$, be independent of $\varepsilon$ satisfying
\begin{align*}
\Lambda \Bigg\{(\varrho,\mathbf{q})\in L^\gamma(\mathbb{T}^3)\times L^{2\gamma/\gamma+1}(\mathbb{T}^3) \, :\, 0<M_1\leq \varrho\leq M_2,\, \mathbf{m}\vert_{\{\varrho=0\}}=0   \Bigg\}=1  
\end{align*}
with constants $0<M_1<M_2$. Furthermore, assume that the following moment estimate
\begin{align}\label{initial}
\int_{L^\gamma_x\times  L^{2\gamma/\gamma+1}_x} \bigg\Vert\frac{1}{2}\frac{\vert  \mathbf{m} \vert^2}{\varrho}+ \frac{ \varrho^\gamma}{\gamma-1}  \bigg\Vert^p_{L^1(\T^3)}\, \mathrm{d}\Lambda (\varrho,\mathbf{m})<\infty
\end{align}
holds for all $1\leq p<\infty$. 
Then the family ${\lbrace \varrho_{\varepsilon}, {\bf m}_{\varepsilon}  \rbrace}_{\varepsilon>0}$
generates, as $\varepsilon \mapsto 0$, a Young measure ${\lbrace \mathcal{V}^{\omega}_{t,x} \rbrace}_{t\in [0,T]; x\in \T^3}$ which is a dissipative measure-valued martingale solution to
the stochastic compressible Euler system \eqref{P1}--\eqref{P2}, in the sense of Definition~\ref{def:dissMartin}, with initial data $\mathcal{V}^{\omega}_{0,x}=\delta_{\varrho_{0}(x), {\bf m}_0(x)}$ $\p$-a.s.
\end{Theorem}
\begin{Remark}
We remark that the assumption $\gamma >\frac 32$ in the above Theorem~\ref{thm:exist} is merely technical and a manifestation of our particular choice of approximate equations, i.e., Navier--Stokes equations, for which existence result is available only for $\gamma >\frac 32$. However, for our purpose, we only need an approximate solution (which works for all $\gamma >1$ ) with appropriate energy bounds, e.g, one may consider ``multipolar fluid'' type approximation \cite{kroner}. In this way, we can exhibit existence of mesure-valued solution to (\ref{P1})--(\ref{P2}) for all $\gamma >1$. 
\end{Remark}

We then establish the following weak (measure-valued)-strong uniqueness principle:
\begin{Theorem}[Weak-Strong Uniqueness] \label{Weak-Strong Uniqueness}
Let $\big[ \big(\Omega,\mathfrak{F}, (\mathfrak{F}_{t})_{t\geq0},\mathbb{P} \big); \mathcal{V}^{\omega}_{t,x}, W \big]$ be a dissipative measure-valued martingale solution to the system \eqref{P1}--\eqref{P2}. On the same stochastic basis $\big(\Omega,\mathfrak{F}, (\mathfrak{F}_{t})_{t\geq0},\mathbb{P} \big)$, let us consider the unique maximal strong pathwise solution to the Euler system \eqref{P1}--\eqref{P2} given by 
$(\bar{\varrho},\bar{{\bf u}},(\mathfrak{t}_R)_{R\in\mn},\mathfrak{t})$ driven by the same cylindrical Wiener process $W$ with the initial data $\bar{\varrho}(0), \bar{\varrho}\bar{\bf u}(0)$ satisfies 
\begin{equation*} 
\mathcal{V}^{\omega}_{0,x}= \delta_{\bar{\varrho}(0,x),(\bar{\varrho}\bar{\bf u})(0,x)}, \,\p-\mbox{a.s.,}\, \mbox{for a.e. } x \in \T^3.
\end{equation*}
Then for a.e. $t\in [0,T]$, $\mathcal{D}(t \wedge \mathfrak{t})=0$, $\p$-a.s., and 
for a.e. $(t,x) \in [0,T] \times \T^3$
\begin{equation*}
\mathcal{V}^{\omega}_{t \wedge \mathfrak{t},x}	= \delta_{\bar{\varrho}(t \wedge \mathfrak{t},x), (\bar{\varrho}\bar{\bf u})(t \wedge \mathfrak{t},x)}, \,\p-\mbox{a.s.}
\end{equation*}
\end{Theorem}


\section{Proof of Theorem~\ref{thm:exist}}
\label{proof1}

The crucial argument used in this proof is based on a compactness method where one needs uniform estimates to demonstrate tightness results and this yields the convergence of the approximate sequence on another probability space and the existence of dissipative measure-valued martingale solution follows, thanks to Young measure theory. Observe that the existence of a pathwise solution (typically obtained by  Gy\"{o}ngy-Krylov's characterization of convergence in probability) seems not possible due to the lack of pathwise uniqueness for the approximate problem (\ref{P1NS})--(\ref{P2NS}).

Let us first note that, thanks to the Theorem~\ref{thm:Romeo1}, the existence of finite energy weak martingale solution of stochastic Navier--Stokes system (\ref{P1NS})--(\ref{P2NS})
$$\big[ \big(\Omega,\mathfrak{F}, (\mathfrak{F}_{{\varepsilon},t})_{t\geq0},\mathbb{P} \big);\varrho_{\varepsilon},\mathbf{u}_{\varepsilon}, W\big]$$
is well established. As alluded to before, lack of pathwise uniqueness for (\ref{P1NS})--(\ref{P2NS}) necessitates the filtration to depend on $\varepsilon$, while the probability space and the Brownian motion can be chosen independent of $\varepsilon$.
The functions $\varrho_\varepsilon$ and $\bfu_\varepsilon$ satisfy the energy inequality, i.e. for any $1 \le p<\infty$ we have 
\begin{align}
&\stred\bigg[\sup_{0\leq t\leq T}\int_{\mt}\Big(\frac{1}{2}\frac{|{\bf m}_\varepsilon|^2}{\varrho_\varepsilon}+\frac{1}{\gamma-1}\varrho^\gamma_\varepsilon\Big)\,\dif x\label{eq:apriorivarepsilon}\\
& \qquad + \varepsilon \int_0^T\int_{\mt}\delta|\nabla_x\bu_\varepsilon|^2+(\lambda+\delta)|\diver_x\bu_\varepsilon|^2\,\dif x\,\dif s\bigg]^p\nonumber\\
& \qquad \qquad \leq \,C_p\bigg(1+\int_{L^\gamma\times L^\frac{2\gamma}{\gamma+1}}\bigg\|\frac{1}{2}\frac{|{\bf m}|^2}{\varrho}+\frac{1}{\gamma-1}\varrho^\gamma\bigg\|_{L^1}^p\,\dif \Lambda(\varrho,{\bf m})\bigg)\leq C(p,T).\nonumber
\end{align}
For a proof of \eqref{eq:apriorivarepsilon}, we refer to the paper by Breit $\&$ Hofmanova \cite{BrHo}. The above relation leads to the following uniform bounds
\begin{align}
 \bfu_\varepsilon&\in L^{p}(\Omega;L^2(0,T;W^{1,2}( \mathbb T^3))),\label{apv}\\
\frac{{\bf m}_\varepsilon}{\sqrt{ \varrho_\varepsilon}} &\in L^{p}(\Omega;L^\infty(0,T;L^2( \mathbb T^3))),\label{aprhov}\\
 \varrho_\varepsilon&\in L^{p}(\Omega;L^\infty(0,T;L^\gamma( \mathbb T^3))),\label{aprho}\\
  {\bf m}_\varepsilon &\in L^{p}(\Omega;L^\infty(0,T;L^\frac{2\gamma}{\gamma+1}( \mathbb T^3))),\label{estrhou2}\\
\frac{{\bf m}_\varepsilon \otimes{\bf m}_\varepsilon}{\varrho_\varepsilon}&\in L^p(\Omega;L^2(0,T;L^\frac{6\gamma}{4\gamma+3}(\mt))).\label{est:rhobfu22}
\end{align}

\subsection{Compactness and almost sure representations}
\label{subsec:compactness}

It is well known that, without assuming any topological structure on $\Omega$, establishing a result of compactness in the probability variable ($\omega$-variable) is a non-trivial task. In what follows, to obtain strong (a.s.) convergence in the $\omega$-variable, we make use of Skorokhod-Jakubowski's representation theorem (cf. \cite{Jakubowski}), linked to tightness of probability measures and a.s. representations of random variables in quasi-Polish spaces adapted to this situation. Note that since our path spaces are not Polish, the classical Skorokhod theorem is not suitable; instead we use the Jakubowski-Skorokhod representation theorem \cite{Jakubowski}.

To establish the tightness of the laws generated by the approximations, we first denote 
the path space $\mathcal{X}$ to be the product of the following spaces:
\begin{align*}
\mathcal{X}_\varrho&=C_w([0,T];L^\gamma(\mt)),&\mathcal{X}_\bu&=\big(L^2(0,T;W^{1,2}(\mt)),w\big),\\
\mathcal{X}_{\varrho\bu}&=C_w([0,T];L^\frac{2\gamma}{\gamma+1}(\mt)),&\mathcal{X}_W&=C([0,T];\mathfrak{U}_0),\\
\mathcal{X}_{C} &= \big(L^{\infty}(0,T; \mathcal{M}_b(\T^3)), w^* \big),
&\mathcal{X}_{P} &= \big(L^{\infty}(0,T; \mathcal{M}_b(\T^3)), w^* \big),\\
\mathcal{X}_{E} &= \big(L^{\infty}(0,T; \mathcal{M}_b(\T^3)), w^* \big), &\mathcal{X}_{\mathcal{V}} &= \big(L^{\infty}((0,T)\times \T^3; \mathcal{P}(\R^4)), w^* \big),\\
\mathcal{X}_M&=\big(W^{\alpha, 2}(0,T; W^{-m,2}(\T^3)), w \big), & \mathcal{X}_N&=C([0,T]; \R),\\
\mathcal{X}_{D}&= \big(L^{\infty}(0,T; \mathcal{M}_b(\T^3)), w^* \big).
\end{align*}
Let us denote by $\mu_{\varrho_\varepsilon}$, $\mu_{\bu_\varepsilon}$, $\mu_{\varrho_\varepsilon\bu_\varepsilon}$, and $\mu_{W_\varepsilon}$ respectively, the law of $\varrho_\varepsilon$, $\bu_\varepsilon$, $\varrho_\varepsilon\bu_\varepsilon$, and $W_\varepsilon$ on the corresponding path space. Moreover, let $\mu_{M_\varepsilon}$, and $\mu_{N_\varepsilon}$ denote the law of two martingales $M_\varepsilon:=\int_0^t {\tn{G}}(\varrho_\varepsilon, \textbf{m}_\varepsilon )\,\Dif W$, and $N_\varepsilon:=\sum_{k\geq1}\int_0^t\intTor{ \vu_{\varepsilon} \cdot {\vc{G}_k}(\varrho_\varepsilon, \textbf{m}_\varepsilon ) } \,\Dif W$ on the corresponding path spaces. Furthermore, let $\mu_{C_\varepsilon}$, $\mu_{D_\varepsilon}$, $\mu_{E_\varepsilon}$, $\mu_{P_\varepsilon}$, and $\mu_{{\mathcal{V}}_\varepsilon}$ denote the law of 
\begin{align*}
& C_\varepsilon:= \frac{\textbf{m}_\varepsilon\otimes \textbf{m}_\varepsilon}{\varrho_\varepsilon}, \quad D_{\varepsilon}:= \sum_{k\geq1}\rho_{\varepsilon}^{-1}| \mathbf{G}_k (\rho_{\varepsilon}, \rho_{\varepsilon} {\bf v}_{\varepsilon}) |^2, \\
& E_\varepsilon :=\frac{1}{2}\frac{|{\bf m}_\varepsilon|^2}{\varrho_\varepsilon}+\frac{1}{\gamma-1}\varrho^\gamma_\varepsilon, \quad P_\varepsilon=  \varrho^\gamma_\varepsilon, \quad {\mathcal{V}}_\varepsilon := \delta_{[\varrho_\varepsilon, \textbf{m}_\varepsilon]},
\end{align*}
respectively, on the corresponding path spaces. Finally, let $\mu^\varepsilon$ denotes joint law of all the variables on $\mathcal{X}$. To proceed further, it is necessary to establish tightness of $\{\mu^\varepsilon;\,\varepsilon\in(0,1)\}$. To this end, we observe that tightness of $\mu_{W_\ep}$ is immediate. So we show tightness of other variables.

\begin{Proposition}\label{prop:rhotight}
The sets $\{\mu_{\varrho_\varepsilon};\,\varepsilon\in(0,1)\}$, $\{\mu_{\bu_\varepsilon};\,\varepsilon\in(0,1)\}$, and $\{\mu_{\varrho_\varepsilon\bu_\varepsilon};\,\varepsilon\in(0,1)\}$ are tight on $\mathcal{X}_\varrho$, $\mathcal{X}_\bu$, and $\mathcal{X}_{\varrho\bu}$ respectively. 
\end{Proposition}

\begin{proof}
For the details of this proof, we refer to \cite{BrHo}. 
\end{proof}

\begin{Proposition}\label{rhoutight14}
The set $\{\mu_{C_{\varepsilon}}, \mu_{D_{\varepsilon}}, \mu_{E_{\varepsilon}},  \mu_{P_{\varepsilon}};\,\varepsilon\in(0,1),\, k\ge 1\}$ is tight on $ \mathcal{X}_{C} \times \mathcal{X}_{D}  \times \mathcal{X}_{E} \times \mathcal{X}_{P}$.
\end{Proposition}

\begin{proof}
These follow immediately from the \emph{a priori} bounds using the fact that all bounded sets in $L^{\infty}(0,T; \mathcal{M}_b(\T^3))$ are relatively compact with respect to the weak-$*$ topology.
\end{proof}

\begin{Proposition}\label{rhoutight1}
The set $\{\mu_{{\mathcal{V}}_{\varepsilon}};\,\varepsilon\in(0,1)\}$ is tight on $\mathcal{X}_{\mathcal{V}}$.
\end{Proposition}

\begin{proof}
The aim is to apply the compactness criterion in $\big(L^{\infty}((0,T)\times \T^3; \mathcal{P}(\R^4)), w^* \big)$. Define the set
\begin{align*}
B_R:= \Big\lbrace {\mathcal{V}} \in \big(L^{\infty}((0,T)\times \T^3; \mathcal{P}(\R^4)), w^* \big); 
\int_0^T \int_{\T^3} \int_{\R^4} \Big(|\xi_1|^{\gamma} + |\xi_2|^{\frac{2\gamma}{\gamma+1}} \Big) \,d{\mathcal{V}}_{t,x}(\xi)\,dx\,dt \le R    \Big\rbrace,
\end{align*}
which is relatively compact in $\big(L^{\infty}((0,T)\times \T^3; \mathcal{P}(\R^4)), w^* \big)$. Note that
\begin{align*}
\mathcal{L}[{\mathcal{V}}_{\varepsilon}](B^c_R)&=
\p\Bigg( \int_0^T \int_{\T^3} \int_{\R^4} \Big(|\xi_1|^{\gamma} + |\xi_2|^{\frac{2\gamma}{\gamma+1}} \Big) \,d{\mathcal{V}}_{t,x}(\xi)\,dx\,dt > R  \Bigg) \\
&= \p\Bigg(\int_0^T \int_{\T^3} \Big(|\varrho_\varepsilon|^{\gamma} + |\textbf{m}_{\varepsilon}|^{\frac{2\gamma}{\gamma+1}} \Big) \,dx\,dt >R \Bigg)
 \le \frac1R \E\Big[\|\varrho_\varepsilon \|_{L^\gamma}^{\gamma} + \| \textbf{m}_{\varepsilon}\|_{L^\frac{2\gamma}{\gamma+1}}^{\frac{2\gamma}{\gamma+1}}\Big] \le \frac CR.
\end{align*}
The proof is complete.
\end{proof}

\begin{Proposition}\label{rhoutight12}
The set $\{\mu_{M_{\varepsilon}};\,\varepsilon\in(0,1)\}$ is tight on $\mathcal{X}_{M}$.
\end{Proposition}

\begin{proof}
First note that $M_{\varepsilon}= \int_0^t\,\mathbb{G}(\varrho_\varepsilon,\varrho_\varepsilon\bu_\varepsilon) \,\dif W(s) \in L^p\big(\Omega; W^{\alpha, 2}(0,T; W^{-m,2}(\T^3))\big)$, and hence tightness follows from weak compactness.
\end{proof}

\begin{Proposition}\label{rhoutight13}
The set $\{\mu_{N_{\varepsilon}};\,\varepsilon\in(0,1)\}$ is tight on $\mathcal{X}_{N}$.
\end{Proposition}

\begin{proof}
First observe that, for each $\varepsilon$, $N_\varepsilon(t)=\sum_{k\geq1}\int_0^t\intTor{ \vu_{\varepsilon} \cdot {\vc{G}_k}(\varrho_\varepsilon, \textbf{m}_\varepsilon ) } \,\Dif W$ is a square integrable martingale. Note that for $r > 2$
\begin{align*}
\E\Big[ \Big|\sum_{k\ge 1} \int_s^t \int_{\T^3}\vu_{\varepsilon} \cdot {\vc{G}_k}(\varrho_\varepsilon, \textbf{m}_\varepsilon )\Big|^r  \Big] &\le \E\Big[ \int_s^t \sum_{k=1}^{\infty} \Big|\int_{\T^3}\vu_{\varepsilon} \cdot {\vc{G}_k}(\varrho_\varepsilon, \textbf{m}_\varepsilon)\Big|^2  \Big]^{r/2} \\
& \le |t-s|^{r/2}\, \Big(1 + \E \Big[\sup_{0\le t \le T} \| \sqrt{\varrho_\varepsilon} u_{\varepsilon}\|^r_{L^2} \Big]\Big)
\le C|t-s|^{r/2},
\end{align*}
and the Kolmogorov continuity criterion (cf. Lemma~\ref{lemma01}) applies. This in particular implies that, for some $\alpha>1$
$$
\sum_{k\geq1}\int_0^t\intTor{ \vu_{\varepsilon} \cdot {\vc{G}_k}(\varrho_\varepsilon, \textbf{m}_\varepsilon ) } \,\Dif W \in L^r(\Omega; C^{\alpha}(0,T; \R)).
$$
Therefore, tightness of law follows from the compact embedding of $C^{\alpha}$ into $C^0$.
\end{proof}

Combining all the informations obtained from Proposition~\ref{prop:rhotight}, Proposition~\ref{rhoutight14}, Proposition~\ref{rhoutight1}, Proposition~\ref{rhoutight12}, and Proposition~\ref{rhoutight13}, we conclude that
\begin{Corollary}
The set $\{\mu^\varepsilon;\,\varepsilon\in(0,1)\}$ is tight on $\mathcal{X}$. 
\end{Corollary}

At this point, we are ready to apply Jakubowski-Skorokhod representation theorem (see also Brzezniak et. al. \cite{BrzezniakHausenblasRazafimandimby}) to extract a.s convergence on a new probability space. In what follows, passing to a weakly convergent subsequence $\mu^{\varepsilon}$ (and denoting by $\mu$ the limit law) we infer the following result:

\begin{Proposition}\label{prop:skorokhod1}
There exists a subsequence $\mu^\varepsilon$ (not relabelled), a probability space $(\tilde\Omega,\tilde\mf,\tilde\prst)$ with $\mathcal{X}$-valued Borel measurable random variables $(\tilde\varrho_{\varepsilon},\tilde{\bf m}_\varepsilon,\tilde{\textbf{u}}_\varepsilon, \tilde W_\varepsilon, \tilde C_{\varepsilon}, \tilde D_{\varepsilon}, \tilde E_{\varepsilon}, \tilde P_{\varepsilon}, \tilde M_{\varepsilon}, \tilde N_{\varepsilon}, \tilde \nu_{\varepsilon})$, $\varepsilon \in (0,1)$, and  $(\tilde\varrho,\tilde{\bf m},\tilde{\textbf{u}},\tilde W, \tilde C, \tilde D, \tilde E, \tilde P, \tilde M, \tilde N, \tilde \nu)$ such that 
\begin{enumerate}
 \item [(1)]the law of $(\tilde\varrho_{\varepsilon},\tilde{\bf m}_\varepsilon, \tilde{\textbf{u}}_\varepsilon, \tilde W_\varepsilon, \tilde C_{\varepsilon}, \tilde D_{\varepsilon}, \tilde E_{\varepsilon}, \tilde P_{\varepsilon}, \tilde M_{\varepsilon}, \tilde N_{\varepsilon}, \tilde \nu_{\varepsilon})$ is given by $\mu^\varepsilon$, $\varepsilon\in(0,1)$,
\item [(2)]the law of $(\tilde\varrho,\tilde{\bf m},\tilde{\textbf{u}},\tilde W, \tilde C, \tilde D, \tilde E, \tilde P, \tilde M, \tilde N, \tilde \nu)$, denoted by $\mu$, is a Radon measure,
\item [(3)]$(\tilde\varrho_{\varepsilon},\tilde{\bf m}_\varepsilon, \tilde{\textbf{u}}_\varepsilon, \tilde W_\varepsilon, \tilde C_{\varepsilon}, \tilde D_{\varepsilon}, \tilde E_{\varepsilon}, \tilde P_{\varepsilon}, \tilde M_{\varepsilon}, \tilde N_{\varepsilon}, \tilde \nu_{\varepsilon})$ converges $\,\tilde{\prst}$-almost surely to \\$(\tilde\varrho,\tilde{\bf m},\tilde{\textbf{u}}, \tilde W, \tilde C, \tilde D, \tilde E, \tilde P, \tilde M, \tilde N, \tilde \nu)$ in the topology of $\mathcal{X}$, i.e.,
\begin{align*}
&\tilde\varrho_\varepsilon \rightarrow \tilde\varrho \,\, \text{in}\, \,C_w([0,T]; L^{\gamma}(\T^3)), \quad
&\tilde{\textbf{m}}_\varepsilon &\rightarrow \tilde{\textbf m} \,\, \text{in}\, \,C_w([0,T]; L^{\frac{2\gamma}{\gamma+1}}(\T^3)),\\
&\tilde{\textbf{u}}_\varepsilon \rightarrow \tilde{\textbf u} \,\, \text{weak in}\, \,L^2( 0,T; W^{1,2}(\T^3)), \quad
&\tilde W_\varepsilon &\rightarrow \tilde W \,\, \text{in}\, \,C([0,T]; \mathfrak{U}_0)),\\
& \tilde C_\varepsilon \rightarrow \tilde C \,\, \text{weak-$*$ in}\, \, L^{\infty}(0,T; \mathcal{M}_b(\T^3)), \quad 
&\tilde N_\varepsilon &\rightarrow \tilde N \,\, \text{in}\, \, C([0,T]; \R),\\
& \tilde E_\varepsilon \rightarrow \tilde E \,\, \text{weak-$*$ in}\, \, L^{\infty}(0,T; \mathcal{M}_b(\T^3)), \quad 
& \tilde P_\varepsilon &\rightarrow \tilde P \,\, \text{weak-$*$ in}\, \, L^{\infty}(0,T; \mathcal{M}_b(\T^3)),\\
& \tilde \nu_\varepsilon \rightarrow \tilde \nu \,\, \text{weak-$*$ in}\, \, L^{\infty}((0,T)\times \T^3; \mathcal{P}(\R^4)), \quad 
& \tilde M_\varepsilon &\rightarrow \tilde M \,\, \text{weak in}\, \, W^{\alpha, 2}(0,T; W^{-m,2}(\T^3)),\\
& \tilde D_\varepsilon \rightarrow \tilde D \,\, \text{weak-$*$ in}\, \, L^{\infty}(0,T; \mathcal{M}_b(\T^3)).
\end{align*}
\item [(4)] For any Carath\'{e}odory function $H=H(t,x,\varrho,\textbf m)$, where $(t,x)\in (0,T)\times \T^3$ and $(\varrho,\textbf m) \in \R^4$, satisfying for some $p, q$ the growth condition
\begin{align*}
|H(t,x,\varrho,\textbf m)| \le 1 + |\varrho|^{p} + |\textbf m|^q,
\end{align*}
uniformly in $(t,x)$. Then we have $\tilde\p$-a.s.
$$
H(\tilde\varrho_\varepsilon, \tilde{\textbf{m}}_\varepsilon) \rightarrow \overline{H(\tilde\varrho, \tilde{\textbf m})}\,\, \text{in}\,\, L^r((0,T)\times\T^3),\,\, \text{for all}\,\, 1<r\le\frac{\gamma}{p}\wedge \frac{2\gamma}{q(\gamma+1)}.
$$
\end{enumerate}
\end{Proposition}

\begin{proof}
Proof of the items $(1)$, $(2)$, and $(3)$ directly follow from Jakubowski-Skorokhod representation theorem. For the proof of the item $(4)$, we refer to the Lemma~\ref{lemma001}.
\end{proof}

We remark that the energy inequality \eqref{eq:apriorivarepsilon} continues to hold on the new probability space, thanks to the equality of joint laws. In other words, all the \emph{a priori} estimates \eqref{apv}--\eqref{est:rhobfu22} also hold for the new random variables.

\subsubsection{Passage to the limit}
Given the above convergences, our aim is to pass to the limit in approximate equations \eqref{P1NS}--\eqref{P2NS}, and energy inequality \eqref{en_01}. We first show that the approximations $\tilde \vr_\ep, \tilde u_{\varepsilon}$ solve equations \eqref{P1NS}--\eqref{P2NS} on the new probability space $(\tilde\Omega,\tilde\mf,\tilde\prst)$.
To that context, let $(\tilde{\mf}_t^\varepsilon)$ and $(\tilde{\mf}_t)$, respectively, be the $\tilde{\prst}$-augmented canonical filtration of the processes $(\tilde\varrho_\varepsilon,\tilde{\bf m}_\varepsilon,\tilde{W}_\varepsilon,\tilde M_{\varepsilon},\tilde N_{\varepsilon})$ and $(\langle {\mathcal{\tilde V}^{\omega}_{t,x}}; \tilde \varrho \rangle,\langle {\mathcal{\tilde V}^{\omega}_{t,x}}; \tilde {\textbf m} \rangle,\tilde{W}, \tilde{M}, \tilde{N})$, respectively, that is
\begin{equation*}
\begin{split}
\tilde{\mf}_t^\varepsilon&=\sigma\big(\sigma\big(\bfr_t\tilde\varrho_\varepsilon,\,\bfr_t\tilde{\bf m}_\varepsilon,\,\bfr_t \tilde{W}_\varepsilon,\bfr_t \tilde M_{\varepsilon}, \bfr_t \tilde N_{\varepsilon}\big)\cup\big\{N\in\tilde{\mf};\;\tilde{\prst}(N)=0\big\}\big),\quad t\in[0,T],\\
\tilde{\mf}_t&=\sigma\big(\sigma\big(\bfr_t\langle {\mathcal{\tilde V}^{\omega}_{t,x}}; \tilde \varrho \rangle, \,\bfr_t\langle {\mathcal{\tilde V}^{\omega}_{t,x}}; \tilde {\textbf m} \rangle,\,\bfr_t\tilde{W}, \bfr_t\tilde{M}, \bfr_t\tilde{N}\big)\cup\big\{N\in\tilde{\mf};\;\tilde{\prst}(N)=0\big\}\big),\quad t\in[0,T],
\end{split}
\end{equation*}
where we denote by $\bfr_t$ the operator of restriction to the interval $[0,t]$ acting on various path spaces.

\begin{Proposition}\label{prop:martsol}
For every $\varepsilon\in(0,1)$, $\big((\tilde{\Omega},\tilde{\mf},(\tilde{\mf}_{\varepsilon,t}),\tilde{\prst}),\tilde\varrho_\varepsilon,\tilde{\bu}_\varepsilon,\tilde{W}\big)$ is a finite energy weak martingale solution to \eqref{P1NS}--\eqref{en_01} with the initial law $\Lambda_\varepsilon$. 
\end{Proposition}

\begin{proof}
Proof of the above proposition directly follows form the Theorem 2.9.1 of the monograph by Breit et. al. \cite{BrFeHobook}.
\end{proof}

\noindent We note that the above proposition implies that the new random variables satisfy the following equations and the energy inequality on the new probability space
\begin{itemize}
\item for all $\phi\in C^\infty(\mathbb{T}^3)$ and $\bm{\varphi}\in C^\infty(\mathbb{T}^3)$ we have
\begin{equation}\label{eq:energyt}
\begin{aligned}
\langle \tilde\varrho_{\varepsilon}(t), \phi\rangle &= \langle\tilde\varrho_{\varepsilon}(0) , \phi\rangle - \int_0^{t}\langle \tilde {\bf m}_{\varepsilon}, \nabla_x \phi\rangle\mathrm{d}s 
\\
\langle \tilde {\bf m}_{\varepsilon}(t), \bm\varphi\rangle &= \langle \tilde {\bf m}_{\varepsilon}(0), \bm{\varphi}\rangle - \int_0^{t} \bigg\langle \frac{\tilde {\bf m}_{\varepsilon}\otimes\tilde {\bf m}_{\varepsilon}}{\tilde \varrho_{\varepsilon}}, \nabla_x  \bm{\varphi} \bigg \rangle\,\mathrm{d}s
+ {\varepsilon}\,\int_0^{t} \langle\mathbb{S}(\nabla_x \tilde{\mathbf{u}}_{\varepsilon})\, \nabla_x  \bm{\varphi}\rangle\mathrm{d}s 
\\&
- \int_0^{t}\langle  \tilde \varrho_{\varepsilon}^\gamma , \mathrm{div}_x\,\bm{\varphi}\rangle\,\mathrm{d}s
+\int_0^{t}\langle\mathbb{G}(\tilde \varrho_{\varepsilon},\tilde {\bf m}_{\varepsilon}), \bm{\varphi}\rangle\, \mathrm{d}\tilde W_{\varepsilon}
\end{aligned}
\end{equation}
$\mathbb{P}$-a.s. for all $t\in[0,T]$,
\item the energy inequality\index{energy inequality} 
\begin{align}\label{EI2t''}
\begin{aligned}
&
-\int_0^T \partial_t \psi \int_{\T^3} \bigg[ \frac{1}{2} \frac{ | \tilde {\bf m}_{\varepsilon} |^2 }{\tilde\varrho_{\varepsilon}} + \frac{\tilde \varrho_{\varepsilon}^\gamma}{\gamma-1} \bigg] \,dx\,dt 
+ {\varepsilon}\, \int_0^T \psi \int_{\T^3}  \mathbb{S} (\nabla_x  \tilde {\bf u}_{\varepsilon}): \nabla_x  \tilde {\bf u}_{\varepsilon} \,dx\,dt \\
&\qquad \leq \psi(0)\int_{\T^3} \bigg[ \frac{1}{2}\frac{|\tilde {\bf m}_{\varepsilon}(0)|^2}{\tilde\varrho_{\varepsilon}(0)}  + \frac{ \tilde \varrho_{\varepsilon}^\gamma(0)}{\gamma-1}  \bigg] \dx
+\sum_{k=1}^\infty\int_0^T \psi \bigg(\int_{\T^3}\mathbf{G}_k (\tilde \varrho_{\varepsilon}, \tilde {\bf m}_{\varepsilon})\cdot \tilde{\bf u}_{\varepsilon}\dx\bigg){\rm d} \tilde W_{\varepsilon,k}\\
&\qquad \qquad + \frac{1}{2}\sum_{k = 1}^{\infty}  \int_0^T \psi
\int_{\T^3} \tilde \varrho_{\varepsilon}^{-1}| \mathbf{G}_k (\tilde \varrho_{\varepsilon}, \tilde {\bf m}_{\varepsilon}) |^2 \,dx\,dt
\end{aligned}
\end{align}
holds for all $\psi \in C_c^{\infty}([0,T))$, $\psi \ge 0$, $\mathbb P$-a.s. 
\end{itemize}

\noindent Next we would like to pass to the limit in $\varepsilon$ in \eqref{eq:energyt} and \eqref{EI2t''}. To do this, we first recall that a-priori estimates \eqref{apv}--\eqref{est:rhobfu22} continue to hold for the new random variables. Thus, making use of the item $(5)$ of Proposition~\ref{prop:skorokhod1}, we conclude that $\tilde \p$-a.s., 
\begin{align*}
&\tilde \varrho_\varepsilon \rightharpoonup \langle {\mathcal{\tilde V}^{\omega}_{t,x}}; \tilde \varrho \rangle, \,\,\text{weakly in}\,\, L^{\gamma}((0,T)\times\T^3),\\
&\tilde {\textbf m}_\varepsilon \rightharpoonup \langle {\mathcal{\tilde V}^{\omega}_{t,x}}; \tilde {\textbf m} \rangle, \,\,\text{weakly in}\,\, L^{\frac{2\gamma}{\gamma+1}}((0,T)\times\T^3).
\end{align*}
In order to pass to the limit in the nonlinear terms present in the equations, we first introduce the corresponding concentration defect measures 
\begin{align*}
\tilde \mu_{C}&= \tilde C -\left\langle \mathcal{\tilde V}^{\omega}_{(\cdot, \cdot)}; \frac{\tilde {\bf m}\otimes \tilde {\bf m}}{\tilde\varrho} \right\rangle dxdt, \,\,
&\tilde \mu_{P}= \tilde P -\langle \mathcal{\tilde V}^{\omega}_{(\cdot, \cdot)};p(\tilde\varrho)\rangle dxdt, \\
\tilde \mu_{E} &= \tilde E- \left\langle \mathcal{\tilde V}^{\omega}_{(\cdot, \cdot)}; \frac{1}{2} \frac{|\tilde {\bf m}|^2}{\tilde\varrho} +P(\tilde\varrho) \right \rangle dx, \,\,
&\tilde \mu_{D}= \tilde D -\left\langle \mathcal{\tilde V}^{\omega}_{(\cdot, \cdot)}; \sum_{k \geq 1} \frac{ |{\bf G}_k (\tilde \varrho, \tilde {\bf m}) |^2 }{\tilde \varrho}\right\rangle dxdt.
\end{align*}
With the help of these concentration defect measures, thanks to the discussion in Subsection~\ref{ym}, we can conclude that $ \mathbb{\tilde P}$-a.s.
\begin{align*}
&\tilde C_\varepsilon \rightharpoonup \left\langle \mathcal{\tilde V}^{\omega}_{(\cdot, \cdot)}; \frac{\tilde {\bf m}\otimes \tilde {\bf m}}{\tilde\varrho} \right\rangle dxdt + \tilde \mu_{C}, \,\, \text{ weak-$*$ in}\, \, L^{\infty}(0,T; \mathcal{M}_b(\T^3)), \\
&\tilde D_\varepsilon \rightharpoonup \left\langle \mathcal{\tilde V}^{\omega}_{(\cdot, \cdot)}; \sum_{k \geq 1} \frac{ |{\bf G}_k (\tilde \varrho, \tilde {\bf m}) |^2 }{\tilde \varrho}\right\rangle dxdt + \tilde \mu_{D}, \,\, \text{weak-$*$ in}\, \, L^{\infty}(0,T; \mathcal{M}_b(\T^3)), \\
&\tilde E_\varepsilon \rightharpoonup \left\langle \mathcal{\tilde V}^{\omega}_{(\cdot, \cdot)}; \frac{1}{2} \frac{|\tilde {\bf m}|^2}{\tilde\varrho} +P(\tilde\varrho) \right \rangle dx + \tilde \mu_{E}, \,\, \text{weak-$*$ in}\, \, L^{\infty}(0,T; \mathcal{M}^+_b(\T^3)), \\
&\tilde P_\varepsilon \rightharpoonup \langle \mathcal{\tilde V}^{\omega}_{(\cdot, \cdot)};p(\tilde\varrho)\rangle dxdt + \tilde \mu_{P}, \,\, \text{weak-$*$ in}\, \, L^{\infty}(0,T; \mathcal{M}_b(\T^3)). 
\end{align*}

\noindent  Regarding the convergence of martingale terms $\tilde M_{\varepsilon}$, appearing in the momentum equation, and $\tilde N_{\varepsilon}$, appearing in the energy inequality, we have following propositions.
\begin{Proposition}
For each $t$, $\tilde N_{\varepsilon}(t) \rightarrow \tilde N(t)$ in $\R$, $\p$-a.s., and $\tilde N(t)$ is a real valued square-integrable continuous martingale with respect to the filtration $(\tilde{\mf}_t)$.
\end{Proposition}

\begin{proof}
Note that, thanks to Proposition~\ref{prop:skorokhod1}, we have the information $\tilde N_{\varepsilon} \rightarrow  \tilde N $ $\p$-a.s. in $C([0,T]; \R)$. To conclude that $\tilde N(t)$ is a martingale, it is enough to show that
$$
\tilde \E[\tilde N(t)| \mathcal{\tilde F}_s] = \tilde N(s),
$$
for all $t,s \in [0,T]$ with $s \le t$. To prove this, it is sufficient to show that
$$
\tilde \E \Big[ L_s(\tilde \Phi) \big(\tilde N(t)-\tilde N(s)\big) \Big]=0,
$$
where we defined $\tilde \Phi = (\tilde \varrho, \tilde {\bf m}, \tilde W,\tilde M,\tilde N)$, and $L_s$ is any bounded continuous functional, depending only on the values of $\tilde \Phi$ restricted to $[0,s]$, on the path space $\underline{\mathcal{X}} :=\mathcal{X}_{\varrho} \times \mathcal{X}_{{\bf m}} \times \mathcal{X}_{W}\times\mathcal{X}_{M}\times\mathcal{X}_{N}$. Now using the information that $\tilde N_{\varepsilon}(t)$ is a martingale, we know that
$$
\tilde \E \Big[ L_s(\tilde \Phi_{\varepsilon}) \big(\tilde N_{\varepsilon}(t)-\tilde N_{\varepsilon}(s)\big) \Big]=0,
$$
for all bounded continuous functional $L_s$ on the same path space, and $\tilde \Phi_{\varepsilon} = (\tilde \varrho_{\varepsilon}, \tilde {\bf m}_{\varepsilon}, \tilde W_{\varepsilon},\tilde M_{\varepsilon},\tilde N_{\varepsilon})$. Note that Proposition~\ref{prop:skorokhod1} reveals that $\tilde \Phi_{\varepsilon} \rightarrow \tilde \Phi$, $\p$-a.s. in the (weak) topology of $\underline{\mathcal{X}}$. This, in particular, implies that $L_s(\tilde \Phi_{\varepsilon}) \rightarrow L_s(\tilde \Phi)$ $\p$-a.s. This information along with higher order moment estimate e.g. $\tilde N_{\varepsilon}(t) \in L^2(\tilde \Omega)$, with the help of Vitali's convergence theorem, we can pass to the limit in $\varepsilon$ to conclude that $\tilde N(t)$ is a martingale. In this way, we loose the structure of the martingale $\tilde N(t)$, which is expected due of lack of sufficient informations.
\end{proof}
\begin{Proposition}
For each $t$ and $r=m+1$, $\tilde M_{\varepsilon}(t) \rightarrow \tilde M(t)$, $\p$-a.s. in the (weak) topology of $W^{-r,2}(\T^3)$. Moreover, $\tilde M(t)$ is also a $W^{-r,2}(\T^3)$-valued martingale with respect to the filtration $(\tilde{\mf}_t)$.
\end{Proposition}

\begin{proof}
As before, thanks to Proposition~\ref{prop:skorokhod1} and the compact embedding given in Lemma~\ref{comp}, we conclude that for each $t$, $ \tilde M_{\varepsilon}(t) \rightarrow \tilde M(t)$, $\p$-a.s. in the topology of $W^{-r,2}(\T^N)$. To show that $\tilde M(t)$ is a martingale, it is sufficient to prove that for all $j\ge1$
$$
\tilde \E \Big[ L_s(\tilde \Phi) \big <\tilde M(t)-\tilde M(s), g_j \big> \Big]=0,
$$
where $g_j$'s are orthonormal basis for the space $W^{-r,2}(\T^N)$, and the bracket $\big<\cdot, \cdot \big>$ denotes inner product in the same space. Again, we use the fact that 
$$
\tilde \E \Big[ L_s(\tilde \Phi_{\varepsilon}) \big<\tilde M_{\varepsilon}(t)-\tilde M_{\varepsilon}(s), g_j \big> \Big]=0,
$$
for all $j \ge 1$. Then we can pass to the limit in $\varepsilon$, as before, to conclude that $\tilde M(t)$ is a martingale. Indeed, we can show uniform integrabilty using BDG inequality:
\begin{align*}
\tilde \E \Big[\big|\big< \tilde M_\varepsilon(t), g_j \big>\big|^p  \Big] &=\tilde \E \bigg[\bigg|\Big< \int_0^t {\tn{G}}(\tilde\varrho_\varepsilon, \tilde {\textbf m}_\varepsilon )\,\Dif \tilde W, g_j \Big>\bigg|^p  \bigg] 
 \le C \tilde \E \Bigg[\sup_{0\le t \le T} \bigg \| \int_0^t  {\tn{G}}(\tilde \varrho_\varepsilon, \tilde {\textbf m}_\varepsilon )\,\Dif \tilde W \bigg\|^p_{W^{-r,2}(\T^3)}\Bigg]  \\
&\le C \tilde \E \Bigg[\bigg(\int_0^T \| {\tn{G}}(\tilde \varrho_\varepsilon, \tilde {\textbf m}_\varepsilon )\|^2_{L_2(\mathfrak{U}, W^{-r,2}(\T^3))} \,ds \bigg)^{p/2}\Bigg] \le C.
\end{align*}
\end{proof}
Note that collating all the above informations, we can pass to the limit in both equations of \eqref{eq:energyt} to infer that
\begin{equation*} 
\begin{aligned}
&\int_{\T^3} \langle \mathcal{\tilde V}^{\omega}_{t,x}; \varrho \rangle \, \varphi (\tau, \cdot) \,dx -\int_{\T^3} \langle \mathcal{\tilde V}^{\omega}_{0,x}; \varrho \rangle\, \varphi (0,\cdot)\,dx 
= \int_{0}^{\tau} \int_{\T^3}  \langle \mathcal{\tilde V}^{\omega}_{t,x}; \textbf{m}\rangle \cdot \nabla_x \varphi \,dx \,dt,
\end{aligned}
\end{equation*}
holds $\tilde{\p}$-a.s., for all $\tau \in[0,T)$, and for all $\varphi \in C^{\infty}(\T^3)$. Moreover,
\begin{equation*} 
\begin{aligned}
&\int_{\T^3} \langle \mathcal{\tilde V}^{\omega}_{\tau,x};\textbf{m}\rangle \cdot \bm{\varphi}(\tau, \cdot) dx - \int_{\T^3} \langle \mathcal{\tilde V}^{\omega}_{0,x};\textbf{m}\rangle \cdot \bm{\varphi}(0,\cdot) dx \\
&\qquad = \int_{0}^{\tau} \int_{\T^3} \left[\left\langle \mathcal{\tilde V}^{\omega}_{t,x}; \frac{\textbf{m}\otimes \textbf{m}}{\varrho} \right\rangle: \nabla_x \bm{\varphi} + \langle \mathcal{\tilde V}^{\omega}_{t,x}; p(\varrho)\rangle \divv_x \bm{\varphi} \right] dx dt \\
&\qquad \qquad+ \int_{\T^3} \bm{\varphi}\,\int_0^{\tau} d\tilde{M}(t) \,dx+ \int_{0}^{\tau} \int_{\T^3}  \nabla_x \bm{\varphi}: d \Big(\tilde{\mu}_C + \tilde{\mu}_P \mathbb{I}\Big),
\end{aligned}
\end{equation*}
holds $\tilde{\p}$-a.s., for all $\tau \in [0,T)$, and for all $\bm{\varphi} \in C^{\infty}(\T^3;\mathbb{R}^3)$, where $\Big(\tilde{\mu}_C + \tilde{\mu}_P \mathbb{I}\Big) \in L^{\infty}_{w^*}\big([0,T]; \mathcal{M}_b({\T^3} )\big)$, $\p$-a.s., is a tensor--valued measure,
Therefore, we conclude that \eqref{first condition measure-valued solution}, and \eqref{second condition measure-valued solution} hold. To conclude \eqref{third condition measure-valued solution}, we proceed as follows. First, note that we can pass to the limit in $\varepsilon$ in \eqref{EI2t''} to obtain the following energy inequality in the new probability space:
\begin{equation}
\label{en_new}
\begin{aligned}
& -\int_0^T \partial_t \psi \Bigg[\int_{\T^3} \left\langle \mathcal{\tilde V}^{\omega}_{t,x}; \frac{1}{2} \frac{|\textbf{m}|^2}{\varrho} +P(\varrho) \right \rangle \,dx + \mathcal{ \tilde D}(t) \Bigg]\,dt 
\le \psi(0) \int_{\T^3} \left\langle \mathcal{\tilde V}^{\omega}_{0,x}; \frac{1}{2} \frac{|\textbf{m}|^2}{\varrho} +P(\varrho) \right \rangle \,dx \\
& \quad + \frac{1}{2} \int_0^{T} \psi \bigg(\int_{\T^3} \sum_{k = 1}^{\infty} \left\langle \mathcal{\tilde V}^{\omega}_{\tau,x};\varrho^{-1}| \mathbf{G}_k (\varrho, {\textbf m}) |^2 \right\rangle \bigg)\, dt\,dx + \frac12\int_0^{T} \psi \int_{\T^3} d \tilde \mu_D + \int_0^{T} \psi\, d{\tilde N}.
\end{aligned}
\end{equation} 
Here $0\le \mathcal{\tilde D}(t):= \tilde \mu_{E}(t)(\T^3)$. Next, we choose a specific test function $\psi_{\tau}: [0,T] \rightarrow \R$ as follows: Fix any $t_0$ and $t$ such that $0<t_0<t <T$. For any $\tau >0$ with $0<t_0 -\tau < t + \tau <T$, let 
$\psi_{\tau}$ be a Lipschitz function that is linear on $[t_0 -\tau, t_0] \cup [t, t+\tau]$ and satisfies
\begin{align*}
\psi_\tau (t)= 
\begin{cases} 
0,  \, & \text{if } t \in [0, t_0 -\tau] \,\,\text{or} \,\, t \in [t+\tau, T] \\
1, \, & \text{if } t \in [t_0, t].
\end{cases}
\end{align*}
Then, via a standard regularization argument, we see that $\psi_\tau$ is an admissible test function in \eqref{en_new}. Inserting this $\psi_\tau$ into \eqref{en_new} gives
\begin{align*}
&\frac{1}{\tau} \int_t^{t + \tau} \Bigg[\int_{\T^3} \left\langle \mathcal{\tilde V}^{\omega}_{t,x}; \frac{1}{2} \frac{|\textbf{m}|^2}{\varrho} +P(\varrho) \right \rangle \,dx + \mathcal{ \tilde D}(t)\Bigg]\,dt \\
& \qquad \le \frac{1}{\tau} \int_{t_0 - \tau}^{t_0} \Bigg[\int_{\T^3} \left\langle \mathcal{\tilde V}^{\omega}_{t,x}; \frac{1}{2} \frac{|\textbf{m}|^2}{\varrho} +P(\varrho) \right \rangle \,dx + \mathcal{ \tilde D}(t)\Bigg]\,dt + \frac12\int_{t_0 - \tau}^{t +\tau} \int_{\T^3} d \tilde \mu_D + \int_{t_0 - \tau}^{t +\tau}  d{\tilde N}(s) \\
& \qquad \quad + \frac{1}{2} \int_{t_0 - \tau}^{t +\tau} \bigg(\int_{\T^3} \sum_{k = 1}^{\infty} \left\langle \mathcal{\tilde V}^{\omega}_{\tau,x};\varrho^{-1}| \mathbf{G}_k (\varrho, {\textbf m}) |^2 \right\rangle \bigg)\, dt\,dx.
\end{align*}
Now taking limit as $\tau \rightarrow 0^+$, we conclude that \eqref{third condition measure-valued solution} holds. Note that for $t_0=0$, we need to slightly modify the test function. In this case we take $\psi_\tau$ which takes the value $1$ in $[0,t]$, linear on $[t, t+\tau]$, and zero otherwise and apply the same argument as before.

Therefore, we are only left with the verifications of \eqref{fourth condition measure-valued solutions}, and item (k) of Definition~\ref{def:dissMartin}. To proceed, we start with the following lemma.

\begin{Lemma}\label{rhoutight1311}
Given a stochastic process $f$, as in item $(k)$ of Definition~\ref{def:dissMartin}
$$
\mathrm{d}f  = D^d_tf\,\mathrm{d}t  + \mathbb{D}^s_tf\,\mathrm{d} \tilde W,
$$
the cross variation with $\tilde M$ is given by 
\begin{align*}
\Big<\hspace{-0.14cm}\Big<f(t),  \tilde M(t)  \Big>\hspace{-0.14cm}\Big>
=\sum_{i,j} \Bigg(\sum_{k = 1}^{\infty}  \int_0^t \left \langle \left\langle \mathcal{\tilde V}^{\omega}_{s,x}; \mathbf{G}_k (\tilde \varrho,\tilde {\bf m})\right\rangle, g_i \right \rangle \, \left \langle \mathbb{D}_t^s f(e_k), g_j \right\rangle \,ds \Bigg) \, g_i \otimes g_j.
\end{align*}
\end{Lemma}

\begin{proof}
First note that, according to Da Prato $\&$ Zabczyk \cite{Daprato}, and the definition of $\tilde M_{\varepsilon} $, we have 
\begin{align*}
\Big<\hspace{-0.14cm}\Big< f(t),  \tilde M_{\varepsilon}(t) \Big>\hspace{-0.14cm}\Big> = \sum_{i,j} \Big<\hspace{-0.14cm}\Big< \left \langle \tilde M_{\varepsilon}(t), g_i \right \rangle , \, \left \langle f(t), g_j \right \rangle \Big>\hspace{-0.14cm}\Big>  \,g_i \otimes g_j.
\end{align*} 
Moreover, 
\begin{equation}
\begin{aligned}
\label{imp_01}
\Big<\hspace{-0.14cm}\Big< \left \langle \tilde M_{\varepsilon}(t), g_i \right \rangle , \, \left \langle f(t), g_j \right \rangle \Big>\hspace{-0.14cm}\Big>  & = \Big<\hspace{-0.14cm}\Big< \sum_{k = 1}^{\infty} \int_0^t  \left \langle \mathbf{G}_k (\tilde \varrho_{\varepsilon},\tilde {\bf m}_{\varepsilon}), g_i \right \rangle \mathrm{d} \tilde W_k \, , \,   \sum_{k = 1}^{\infty} \int_0^t  \left \langle \mathbb{D}^s_tf(e_k), g_j \right \rangle\mathrm{d} \tilde W_k  \Big>\hspace{-0.14cm}\Big>  \\
& = \sum_{k = 1}^{\infty} \int_0^t   \left \langle \mathbf{G}_k (\tilde \varrho_{\varepsilon},\tilde {\bf m}_{\varepsilon}), g_i \right \rangle \, \left \langle \mathbb{D}^s_tf(e_k), g_j \right \rangle \,ds.
\end{aligned} 
\end{equation}
Therefore
\begin{align*}
\Big<\hspace{-0.14cm}\Big< f(t),  \tilde M_{\varepsilon}(t) \Big>\hspace{-0.14cm}\Big> = \sum_{i,j} \Bigg( \sum_{k = 1}^{\infty} \int_0^t  \left \langle \mathbf{G}_k (\tilde \varrho_{\varepsilon},\tilde {\bf m}_{\varepsilon}), g_i \right \rangle \, \left \langle \mathbb{D}^s_tf(e_k), g_j \right \rangle \,ds \Bigg) \,g_i \otimes g_j.
\end{align*}
Observe that \eqref{imp_01} equivalently implies that
\begin{align*}
&\tilde \E \Big[ L_s(\tilde \Phi_{\varepsilon}) \Big( \big<\tilde M_{\varepsilon}(t), g_i \big> \big<f(t), g_j \big> - \big<\tilde M_{\varepsilon}(s), g_i \big> \big<f(s), g_j \big> \\
& \hspace{6cm}- \sum_{k = 1}^{\infty} \int_s^t  \big<\mathbf{G}_k (\tilde \varrho_{\varepsilon}, \tilde {\bf m}_{\varepsilon}), g_i \big> \big<\mathbb{D}^s_tf(e_k), g_j \big> \,ds \Big)\Big]=0.
\end{align*}
We may apply item (4) of Proposition~\ref{prop:skorokhod1}, to the composition $\mathbf{G}_k (\tilde \varrho_{\varepsilon},\tilde {\bf m}_{\varepsilon})$, $k \in \N$. This gives 
$$
\mathbf{G}_k (\tilde \varrho_{\varepsilon},\tilde {\bf m}_{\varepsilon}) \rightharpoonup
\Big \langle \mathcal{\tilde V}^{\omega}_{t,x} ; \mathbf{G}_k (\tilde \varrho,\tilde {\bf m}) \Big \rangle \,\, \mbox{weakly in} \,\,L^q((0,T)\times \T^3),
$$
$\p$-a.s., for some $q>1$. Moreover, for $m>3/2$, we have by Sobolev embedding 
\begin{align*}
\tilde \E \Big[\int_0^T \| \mathbb{G}(\tilde \varrho_{\varepsilon}, \tilde {\bf m}_{\varepsilon}) \|^2_{L_2(\mathfrak{U}; W^{-m,2})}\,dt \Big]\le \tilde \E \Big[\int_0^T (\tilde \varrho_{\varepsilon})_{\T^3} \int_{\T^3} (\tilde \varrho_{\varepsilon} + \tilde \varrho_{\varepsilon} |\tilde {\bf u}_{\varepsilon}|^2)\,dx\,dt \Big]\le c(r).
\end{align*}
This implies that $\p$-a.s.
$$
\mathbf{G}_k (\tilde \varrho,\tilde {\bf m}) \rightharpoonup  \Big \langle \mathcal{\tilde V}^{\omega}_{t,x} ; \mathbf{G}_k (\tilde \varrho,\tilde {\bf m}) \Big \rangle \,\, \mbox{weakly in} \,\,L^2((0,T); W^{-m,2}(\T^3)).
$$ 
This will imply that, thanks to uniform integrabilty, we can pass to the limit in $\ep \rightarrow 0$ to conclude
\begin{align*}
\tilde \E \Big[ L_s(\tilde \Phi) & \Big( \big<\tilde M(t), g_i \big> \big<f(t), g_j \big> - \big<\tilde M(s), g_i \big> \big<f(s), g_j \big> \\
& \hspace{4cm}- \sum_{k = 1}^{\infty} \int_s^t  \big<\left\langle \mathcal{\tilde V}^{\omega}_{s,x}; \mathbf{G}_k (\tilde \varrho,\tilde {\bf m})\right\rangle, g_i \big> \big<\mathbb{D}^s_tf(e_k), g_j \big> \,ds \Big)\Big]=0.
\end{align*}
This finishes the proof the lemma.

\end{proof}

\begin{Lemma}\label{rhoutight131}
The concentration defect $0\le \mathcal{\tilde D}(\tau):= \tilde \mu_{E}(\tau)(\T^3)$ dominates defect measures $\tilde \mu_{C}, \tilde \mu_{D}$, and $\tilde \mu_{P}$ in the sense of Lemma~\ref{lemma001}. More precisely, there exists a constant $C>0$ such that
\begin{equation*} 
 \int_{0}^{\tau} \int_{\T^3} d|\tilde \mu_C| + \int_{0}^{\tau} \int_{\T^3} d|\tilde \mu_D| + \int_{0}^{\tau} \int_{\T^3} d|\tilde \mu_P| \leq C \int_{0}^{\tau} \mathcal{\tilde D}(\tau)\,dt,
\end{equation*}	
for every $\tau \in (0,T)$, $\p$-a.s.
\end{Lemma}

\begin{proof}
Following deterministic argument, we can conclude that $\tilde \mu_{E}$ dominates defect measures $\tilde \mu_{C}, \tilde \mu_P$. To show the dominance of $\tilde \mu_{E}$ over $\tilde \mu_{D}$, observe that by virtue of hypotheses (\ref{FG2}), the function
\[
[\vr, \vc{m}] \mapsto \sum_{k \geq 1} \frac{ |{\bf G}_k (\varrho, {\bf m}) |^2 }{\varrho} \ \mbox{is continuous},
\]
and as such dominated by the total energy
\[
\sum_{k \geq 1} \frac{ |{\bf G}_k (\varrho, {\bf m}) |^2 }{\varrho} \leq c \left(  \varrho  + \frac{| {\bf m} |^2}{\vr}  \right) \leq c \left( \frac{1}{2} \frac{|{\bf m}|^2}{\vr}  + P(\vr) \right) + 1.
\]
Hence, a simple application of the Lemma~\ref{lemma001} finishes the proof of the lemma. 
\end{proof}


\section{Weak-Strong Uniqueness Principle}
\label{proof2}
In this section we prove Theorem~\ref{Weak-Strong Uniqueness} through couple of auxiliary results. We start with the following technical lemma, a variant of \cite[Lemma 3.1]{BrFeHo2015A}, to which we refer for the proof. Note that this lemma plays a pivotal role in the proof of the Proposition~\ref{relen}.

\begin{Lemma} \label{lem}
Let $q$ be a stochastic process on $\StoB$ such that for some $b\in\R$,
\[
q \in C_{\rm weak}([0,T]; W^{-b,2}(\tor)) \cap L^\infty (0,T; L^1(\tor)) \quad \text{$\mathbb{P}$-a.s.},
\]
\begin{equation*} 
\Dif q= \DD q\, \dt + \Dif M.
\end{equation*}
Here $M$ is a continuous square integrable $W^{-b,2}(\tor)$ valued martingale, and $\DD q$ is progressively measurable with
\begin{align} \label{hh2}
\begin{aligned}
\DD q\in &L^p(\Omega;L^1(0,T;W^{-b,a}(\tor)),
\end{aligned}
\end{align}
for some $a>1$ and some $m\in\N$.

Let $w$ be a stochastic process on $\StoB$ satisfying
\[
w \in C([0,T]; W^{b,a'} \cap C (\tor)) \quad \text{$\mathbb{P}$-a.s.},
\]
\begin{equation*} \label{hh3}
\E\bigg[ \sup_{t \in [0,T]} \| w \|_{W^{b,a'} \cap C (\tor)}^p \bigg] < \infty,\ 1 \leq p < \infty,
\end{equation*}
\begin{equation*} \label{rel2}
\Dif w = \DD w + \DS w \, \Dif W.
\end{equation*}
Here $\DD w, \DS w$ are progressively measurable with
\begin{align*} \label{hh4}
\begin{aligned}
\DD w\in L^p&(\Omega;L^1(0,T;W^{b,a'}\cap C(\tor)),\quad \DS w\in L^2(\Omega;L^2(0,T;L_2(\mathfrak U;W^{-m,2}(\tor)))),\\
&\sum_{k\geq1} \int_0^T\|\DS w(e_k)\|^2_{W^{b,a'}\cap C(\tor)}\dt\in L^p(\Omega),\quad 1\leq p<\infty.
\end{aligned}
\end{align*}
Let $Q$ be $[b+2]$-continuously differentiable function satisfying
\begin{equation*} \label{hh5}
\E\bigg[\sup_{t \in [0,T]} \| Q^{(j)} (w) \|_{W^{b,a'} \cap C (\tor)}^p \bigg]< \infty, \quad j = 0,1,2,\quad 1 \leq p < \infty.
\end{equation*}
Then
\begin{equation*} \label{result}
\begin{split}
\Dif \left( \intTor{ q Q(w) } \right)
&= \bigg( \intTor{ \Big[  q  \Big( Q'(w) \DD w  + \frac{1}{2}\sum_{k\geq1} Q''(w)
\left| \DS w (e_k)\right|^2  \Big) \Big] }  +  \left< Q(w) , \DD q \right> \bigg) {\rm d}t
\\
&+ \int_{\T^3}\Big<\hspace{-0.14cm}\Big<f(t),  \tilde M(t)  \Big>\hspace{-0.14cm}\Big>\, {\rm d}t
+ {\rm d}\tn{M},
\end{split}
\end{equation*}
where
\begin{equation*} \label{result1}
\tn{M} = \int_0^t\intTor{ \Big[  q \,Q'(w) \, {\rm d} M + Q(w)  \, {\rm d} M  \Big] }\,{\rm d}s.
\end{equation*}
\end{Lemma}

\subsection{Relative energy inequality}
\label{MEI}

\noindent We proceed further and introduce the \textit{relative energy (entropy)} functional. The commonly used form of the \textit{relative energy} functional in the context of measure-valued solutions to the compressible Euler system reads
\begin{equation}
\label{rell}
\begin{aligned} 
&\mathcal{E}^1_{\mathrm{mv}} \left( \varrho,\textbf{m} \ \Big| \ r, {\bf U} \right)(t)
:=
\int_{\T^3} \Bigg\langle {\mathcal{V}^{\omega}_{t,x}}; \frac{1}{2} \frac{ |\textbf{m}|^2}{\vr} + P(\varrho) \Bigg\rangle \,dx + \mathcal{D}(t)
- \intTor{ \big \langle {\mathcal{V}^{\omega}_{t,x}}; \textbf{m} \big \rangle \,\cdot {\bf U}}
\\
& \hspace{2cm}+ \frac{1}{2}  \intTor{ \big \langle {\mathcal{V}^{\omega}_{t,x}}; \vr \big \rangle \, |{\bf U}|^2 } - \intTor{ \big \langle {\mathcal{V}^{\omega}_{t,x}}; \vr \big \rangle \, P'(r) } - \intTor{ \big[ P'(r) r - P(r) \big] }.
\end{aligned}
\end{equation}
We remark that the above relative energy functional is defined for all $t \in [0,T] \setminus N$, where the set $N$, may depends on $\omega$, has Lebesgue measure zero. To define relative energy functional for all $t \in N$, we consider 
\begin{align} 
&\mathcal{E}^2_{\mathrm{mv}} \left( \varrho,\textbf{m} \ \Big| \ r, {\bf U} \right)(t)
:=
\lim_{\tau \rightarrow 0+} \frac{1}{\tau} \int_t^{t +\tau} \Bigg[\int_{\T^3} \Bigg\langle {\mathcal{V}^{\omega}_{s,x}}; \frac{1}{2} \frac{ |\textbf{m}|^2}{\vr} + P(\varrho) \Bigg\rangle\,dx + \mathcal{D}(s) \Bigg]\,ds  \label{rell1}\\
&- \intTor{ \big \langle {\mathcal{V}^{\omega}_{t,x}}; \textbf{m} \big \rangle \,\cdot {\bf U}}
+ \frac{1}{2}  \intTor{ \big \langle {\mathcal{V}^{\omega}_{t,x}}; \vr \big \rangle \, |{\bf U}|^2 } - \intTor{ \big \langle {\mathcal{V}^{\omega}_{t,x}}; \vr \big \rangle \, P'(r) } - \intTor{ \big[ P'(r) r - P(r) \big] }. \notag
\end{align}
Using above, we define relative energy functional as follows
\begin{equation}
\label{rell2}
\begin{aligned}
\mathcal{E}_{\mathrm{mv}} \left( \varrho,\textbf{m} \ \Big| \ r, {\bf U} \right)(t) :=
\begin{cases}
\mathcal{E}^1_{\mathrm{mv}} \left( \varrho,\textbf{m} \ \Big| \ r, {\bf U} \right)(t), \, &\text{if} \,\,t \in [0,T] \setminus N\\
\mathcal{E}^2_{\mathrm{mv}} \left( \varrho,\textbf{m} \ \Big| \ r, {\bf U} \right)(t), \, &\text{if} \,\,t \in N.
\end{cases}
\end{aligned}
\end{equation}
Note that the above relative energy functional is defined for all $t \in [0,T]$. In what follows, with the help of the above defined relative energy, we derive the relative energy inequality \eqref{relativeEntropy}. Note that, the relative energy inequality is a tool which enables us to compare measure valued solutions with some smooth comparison functions. 
\begin{Proposition}[Relative Energy] 
\label{relen}
Let $\big[ \big(\Omega,\mathfrak{F}, (\mathfrak{F}_{t})_{t\geq0},\mathbb{P} \big); {\mathcal{V}^{\omega}_{t,x}}, W \big]$ be a dissipative measure-valued martingale solution to the system \eqref{P1}--\eqref{P2}.
Suppose $\big(r \,,\,\mathbf{U}  \big)$ be a pair of stochastic processes which are adapted to the filtration $(\mathfrak{F}_t)_{t\geq0}$ and which satisfies
\begin{equation*}
\begin{aligned}
\label{operatorBB}
\mathrm{d}r  &= D^d_tr\,\mathrm{d}t  + \mathbb{D}^s_tr\,\mathrm{d}W,
\\
\mathrm{d}\mathbf{U}  &= D^d_t\mathbf{U}\,\mathrm{d}t  + \mathbb{D}^s_t\mathbf{U}\,\mathrm{d}W,
\end{aligned}
\end{equation*}
with
\begin{align*}\label{eq:smooth}
\begin{aligned}
r \in C([0,T]; W^{1,q}(\T^3)), \ \vc{U} \in C([0,T]; W^{1,q}(\T^3)), \ \quad\text{$\mathbb{P}$-a.s.},\\
\E\bigg[\sup_{t \in [0,T] } \| r \|_{W^{1,q}(\tor)}^2\bigg]^q  + \E\bigg[ \sup_{t \in [0,T] } \| \vc{U} \|_{W^{1,q}(\tor)}^2\bigg]^q \leq c(q) \quad\forall\,\, 2 \leq q < \infty,
\end{aligned}
\end{align*}
\begin{equation} \label{bound001}
0 < \underline{r} \leq r(t,x) \leq \overline{r} \quad\text{$\mathbb{P}$-a.s.}
\end{equation}
Moreover, $r$, $\vc{U}$ satisfy
\begin{align}
&D^d r, D^d \vc{U}\in L^q(\T^3;L^q(0,T;W^{1,q}(\mt))),\quad
\mathbb D^s r,\mathbb D^s \vc{U}\in L^2(\T^3;L^2(0,T;L_2(\mathfrak U;L^2(\tor)))),\nonumber\\
&\bigg(\sum_{k\geq 1}|\mathbb D^s r(e_k)|^q\bigg)^\frac{1}{q},\bigg(\sum_{k\geq 1}|\mathbb D^s \vc{U}(e_k)|^q\bigg)^\frac{1}{q}\in L^q(\T^3;L^q(0,T;L^{q}(\mt))).\label{new}
\end{align}
Then the following \emph{relative energy inequality} holds $\mathbb P$-a.s., for all $t\in (0,T)$:
\begin{equation}
\begin{aligned}
\label{relativeEntropy}
&\mathcal{E}_{\mathrm{mv}} \left(\varrho,\textbf{m} \ \Big| \ r, {\bf U} \right)
(t) \leq
\mathcal{E}_{\mathrm{mv}} \left(\varrho,\textbf{m} \ \Big| \ r, {\bf U} \right)(0) +M_{RE}(t)  + \int_0^t\mathcal{R}_{\mathrm{mv}} \big(\varrho,\textbf{m} \left\vert \right. r, \mathbf{U}  \big)(s)\,\mathrm{d}s
\end{aligned}
\end{equation}
where
\begin{equation}
\begin{aligned}
\label{remainderRE}
&\mathcal{R}_{\mathrm{mv}} \big(\varrho,{\bf m} \left\vert \right. r, \mathbf{U}  \big) 
=
\int_{\mathbb{T}^3}\langle {\mathcal{V}^{\omega}_{t,x}}; \varrho \mathbf{U} - {\bf m} \rangle \cdot[D^d_t  \mathbf{U} + \nabla_x  \mathbf{U} \cdot  \mathbf{U}] \,dx \\
&\quad + \int_{\mathbb{T}^3} \left\langle {\mathcal{V}^{\omega}_{t,x}}; \frac{({\bf m}-\varrho  \mathbf{U})\otimes (\varrho  \mathbf{U}-{\bf m}) }{\varrho} \right\rangle: \nabla_x \textbf{U} \,\mathrm{d}x -  \int_{\T^3} \nabla_x \mathbf{U} : d\mu_m  + \frac12\int_{\T^3} d \mu_e
\\
& \quad +
\int_{\mathbb{T}^3}\big[(r- \big \langle {\mathcal{V}^{\omega}_{t,x}}; \varrho \big\rangle)P''(r)D^d_tr + \nabla_x  P'(r) \cdot (r\mathbf{U}-\big \langle {\mathcal{V}^{\omega}_{t,x}}; {\bf m} \big\rangle)  \big] \,\mathrm{d}x
\\
&\quad +
\int_{\mathbb{T}^3}\big[p(r)  - \big \langle {\mathcal{V}^{\omega}_{t,x}}; p(\varrho)\big\rangle \big]\mathrm{div}_x (\mathbf{U}) \,\mathrm{d}x
+
\frac{1}{2}
\sum_{k\in\mathbb{N}}
\int_{\mathbb{T}^3} \Bigg \langle {\mathcal{V}^{\omega}_{t,x}}; \varrho\bigg\vert \frac{\mathbf{G}_k(\varrho,{\bf m})}{\varrho}  -\mathbb{D}^s_t\mathbf{U}(e_k)  \bigg\vert^2\Bigg\rangle \,\mathrm{d}x \\
& \quad + \frac{1}{2}\sum_{k\geq1} \intTor{ \big \langle {\mathcal{V}^{\omega}_{t,x}}; \vr \big \rangle P'''(r) |\DS r(e_k) |^2 } \, {\rm d}t + \frac{1}{2}\sum_{k\geq1} \intTor{ p''(r) |\DS r(e_k)|^2 } \, {\rm d}t.
\end{aligned}
\end{equation}
Here $M_{RE}$ is a real valued square integrable martingale, and the norm of this martingale depends only on the norms of $r$ and $\mathbf{U}$ in the aforementioned spaces. Moreover, the pressure potential $P$ is defined as the solution of the equation $r P'(r) -P(r)=p(r)$.

\end{Proposition}

\begin{proof}
Note that all the integrals on the right hand side of \eqref{rell} can be explicitly expressed by means of either the energy inequality \eqref{third condition measure-valued solution} or the field equations \eqref{first condition measure-valued solution} and \eqref{second condition measure-valued solution}. Therefore, to compute the right hand side of \eqref{rell}, we make use of Lemma \ref{lem} and the energy inequality (\ref{third condition measure-valued solution}).

\medskip

\noindent {\bf Step 1:}

\medskip

To compute $\Dif \intTor{ \big \langle {\mathcal{V}^{\omega}_{t,x}}; \textbf{m} \big \rangle \,\cdot {\bf U} }$ we recall that $q = \big \langle {\mathcal{V}^{\omega}_{t,x}}; \textbf{m} \big \rangle $ satisfies hypotheses (\ref{hh2})
with some $a < \infty$.  Applying Lemma \ref{lem} we obtain
\begin{equation} \label{I1}
\begin{split}
& \Dif \left( \intTor{ \big \langle {\mathcal{V}^{\omega}_{t,x}}; \textbf{m} \big \rangle \cdot {\bf U} } \right) = \intTor{ \left[ \big \langle {\mathcal{V}^{\omega}_{t,x}}; \textbf{m} \big \rangle \cdot \DD {\bf U}
+ \left\langle {\mathcal{V}^{\omega}_{t,x}}; \frac{\textbf{m}\otimes \textbf{m}}{\varrho} \right\rangle: \nabla_x \textbf{U} + \big \langle {\mathcal{V}^{\omega}_{t,x}}; p(\vr) \big\rangle\, \Div {\bf U} \right] }  {\rm d}t \\
& \qquad +  \sum_{k\geq1}\intTor{ \DS {\bf U}(e_k) \cdot \big \langle {\mathcal{V}^{\omega}_{t,x}}; \vc{G}_k (\varrho,\textbf{m}) \big\rangle}\, {\rm d}t + \int_{\T^3}  \nabla_x {\bf U}: d\mu_m \,dt+ \Dif M_1,
\end{split}
\end{equation}
where 
\[
M_1(t) = \int_0^t \int_{\T^3} \vc{U} \,dx \, dM^1_{E} + \int_0^t \intTor{ \big \langle {\mathcal{V}^{\omega}_{t,x}}; \textbf{m} \big \rangle \cdot \DS \vc{U} } \,\Dif W
\]
is a square integrable martingale. Note that to identify the cross variation in (\ref{I1}), we have used item (k) of the Definition~\ref{def:dissMartin}, and noticing that
\begin{align*}
\Big<\hspace{-0.14cm}\Big<\big \langle {\mathcal{V}^{\omega}_{t,x}}; \textbf{m} \big \rangle,  \vc{U}(t)  \Big>\hspace{-0.14cm}\Big>
&=\sum_{i,j} \Bigg(\sum_{k\geq1}  \int_0^t \left \langle \big \langle {\mathcal{V}^{\omega}_{t,x}}; \vc{G}_k (\varrho,\textbf{m}) \big \rangle, g_i \right \rangle \, \left \langle \mathbb{D}_t^s \vc{U}(e_k), g_j \right\rangle \,ds \Bigg) \, g_i \otimes g_j \\
& =\sum_{k\geq1} \int_0^t \big \langle {\mathcal{V}^{\omega}_{t,x}}; \vc{G}_k (\varrho,\textbf{m}) \big\rangle \cdot \DS {\bf U}(e_k) \, {\rm d}s.
\end{align*}

\color{black}

\medskip

\noindent
{\bf Step 2:}

\medskip

Similarly, we compute
\begin{equation} \label{I2}
\begin{split}
\Dif \left( \intTor{ \frac{1}{2} \big \langle {\mathcal{V}^{\omega}_{t,x}}; \vr \big \rangle |\vc{U}|^2 } \right) &=
\intTor{ \big \langle {\mathcal{V}^{\omega}_{t,x}}; \textbf{m} \big \rangle \cdot \Grad {\bf U} \cdot {\bf U} }  {\rm d}t 
+  \intTor{ \big \langle {\mathcal{V}^{\omega}_{t,x}}; \vr \big \rangle {\bf U} \cdot  \DD {\bf U} } {\rm d}t \\
& \qquad + \frac{1}{2} \sum_{k\geq1}\intTor{ \big \langle {\mathcal{V}^{\omega}_{t,x}}; \vr \big \rangle |\DS {\bf U}(e_k)|^2 } \ {\rm d}t + {\rm d}M_2,
\end{split}
\end{equation}
where
\[
M_2(t) = \int_0^t \intTor{ \big \langle {\mathcal{V}^{\omega}_{t,x}}; \vr \big \rangle \vc{U} \cdot \DS \vc{U} } \, \Dif W,
\]
\begin{equation} \label{I3}
{\rm d} \left( \intTor{ \left[ P'(r) r - P(r) \right] } \right)
= \intTor{ p'(r) \DD r } \ {\rm d}t + \frac{1}{2}\sum_{k\geq1} \intTor{ p''(r) |\DS r(e_k)|^2 } \, {\rm d}t + \Dif M_3,
\end{equation}
where
\[
M_3(t) = \int_0^t \intTor{ p'(r) \DS r } \,\Dif W,
\]
and, finally,
\begin{equation} \label{I4}
\begin{split}
{\rm d} \left( \intTor{ \big \langle {\mathcal{V}^{\omega}_{t,x}}; \vr \big \rangle P'(r) } \right) &=
 \intTor{  \Grad P'(r) \cdot \big \langle {\mathcal{V}^{\omega}_{t,x}}; \textbf{m} \big \rangle } \ {\rm d}t 
+ \intTor{ \big \langle {\mathcal{V}^{\omega}_{t,x}}; \vr \big \rangle P''(r) \DD r } \ {\rm d}t  \\
& \qquad + \frac{1}{2}\sum_{k\geq1} \intTor{ \big \langle {\mathcal{V}^{\omega}_{t,x}}; \vr \big \rangle P'''(r) |\DS r(e_k) |^2 } \, {\rm d}t
+ \Dif M_4,
\end{split}
\end{equation}
where
\[
M_4(t) = \int_0^t \intTor{ \big \langle {\mathcal{V}^{\omega}_{t,x}}; \vr \big \rangle P''(r) \DS r } \,\Dif W.
\]

\medskip

\noindent
{\bf Step 3:}

Now we can collect (\ref{I1}--\ref{I4}), define the square integrable real valued martingale $M_{RE}(t):= M_1(t) + M_2(t) +M_3(t) +M_4(t) + M^2_E(t)$, and summing up the resulting expressions and adding the sum with (\ref{third condition measure-valued solution}) to obtain (\ref{relativeEntropy}).
\end{proof}


\subsection{Proof of Theorem~\ref{Weak-Strong Uniqueness}}

As we have seen before, the \emph{relative energy inequality} \eqref{relativeEntropy} is a consequence of the \emph{energy inequality} \eqref{third condition measure-valued solution}. Now we wish to use the Proposition~\ref{relen} to prove the following weak (measure-valued)--strong uniqueness principle given by Theorem~\ref{Weak-Strong Uniqueness}.

In what follows, our aim is to apply the relative energy inequality \eqref{third condition measure-valued solution} to the pair
$(r,\mathbf U)=(\bar{\varrho}(\cdot\wedge \mathfrak t_R),\bar{\bfu}(\cdot\wedge \mathfrak t_R))$, where $(\bar{\varrho},\bar{\mathbf{u}},(\mathfrak{t}_R)_{R\in\mathbb{N}},\mathfrak{t})$ is the unique maximal strong pathwise solution to \eqref{P1}--\eqref{P2} which exists by Theorem \ref{thm:main}. Recall that the stopping time $\mathfrak t_R$ announces the blow-up and satisfies
\begin{equation*}
\sup_{t\in[0,\mathfrak{t}_R]}\|\bar{\vu}(t)\|_{1,\infty}\geq R\quad \text{on}\quad [\mathfrak{t}<T] ;
\end{equation*}
Moreover, $(r,\mathbf U)=(\bar{\varrho},\bar{\bfu})$ satisfies an equation of the form \eqref{operatorBB}, where
\begin{align*}
D^d_tr  = -\mathrm{div}_x (\bar{\varrho}\bar{\mathbf{u}}),\quad \mathbb{D}^s_tr=0,
\quad
D^d_t\mathbf{U}  = -\bar{\mathbf{u}}\cdot\nabla_x \bar{\mathbf{u}} -\frac{1}{\bar{\varrho}}\nabla_x  p(\bar{\varrho}),
\quad
\mathbb{D}^s_t\mathbf{U}  = \frac{1}{\bar{\varrho}}\mathbb{G} (\bar{\varrho},\bar{{\bf m}}).
\end{align*}
It is easy to see that
\eqref{bound001} and \eqref{new} are satisfied for $t\leq \mathfrak t_R$, thanks to Theorem \ref{thm:main} and \eqref{FG2}. Moreover, the lower bound for $\bar{\varrho}$ is a consequence of standard maximum principle.
So, \eqref{relativeEntropy} holds and we can now deduce that for every $t \in [0,T]$ and $R\in\mathbb{N}$,
\begin{equation}
\begin{aligned}
\label{relativeEntropy1}
&\mathcal{E}_{\mathrm{mv}} \left( \varrho,\textbf{m} \ \Big| \ \bar{\varrho}, \bar{\mathbf{u}} \right)  
(t \wedge \mathfrak{t}_R)  \\
& \qquad \qquad \leq
\mathcal{E}_{\mathrm{mv}} \left( \varrho,\textbf{m} \ \Big| \ \bar{\varrho}, \bar{\mathbf{u}} \right)(0) + M_{RE}(t \wedge \mathfrak{t}_R)  + \int_0^{t \wedge \mathfrak{t}_R} \mathcal{R}_{\mathrm{mv}} \big(\varrho,{\bf m} \left\vert \right. \bar{\varrho}, \bar{\mathbf{u}}  \big) (s)\,\mathrm{d}s,
\end{aligned}
\end{equation}
where after a standard manipulation of terms in \eqref{remainderRE}, as in \cite{Fei01}, we obtain
\begin{equation}
\begin{aligned}
\label{remainderRE1}
\mathcal{R}_{\mathrm{mv}} \big(\varrho,{\bf m} & \left\vert \right. \bar{\varrho}, \bar{\mathbf{u}}  \big) 
=\,  \int_{\mathbb{T}^3} \left\langle {\mathcal{V}^{\omega}_{t,x}}; \left|\frac{(\textbf{m}-\varrho \bar{\mathbf{u}})\otimes (\varrho \bar{\mathbf{u}}-\textbf{m}) }{\varrho}\right| \right\rangle |\nabla_x \bar{\mathbf{u}}| \,dx  \\
&-
\int_{\mathbb{T}^3}\langle {\mathcal{V}^{\omega}_{t,x}}; |p(\varrho)-p'(\bar{\varrho})(\varrho-\bar{\varrho})-p(\bar{\varrho})|\rangle |\divv_x \bar{\mathbf{u}}| \,dx\\
&+
\frac{1}{2}
\sum_{k\in\mathbb{N}}
\int_{\mathbb{T}^3} \Bigg \langle {\mathcal{V}^{\omega}_{t,x}}; \varrho\bigg\vert \frac{\mathbf{G}_k(\varrho,\varrho \mathbf{u})}{\varrho}  -\frac{\mathbf{G}_k(\bar{\varrho}, \bar{\varrho} \bar{\mathbf{u}})}{\bar{\varrho}}  \bigg\vert^2 \Bigg \rangle \,\mathrm{d}x 
+  \int_{\T^3} |\nabla_x \bar{\mathbf{u}}|\cdot d|\mu_m| +  \frac12 \int_{\T^3} d|\mu_c|.
\end{aligned}
\end{equation}
Since $\bar{\textbf{u}}$ has compact support we can control the terms $|\nabla_x \bar{\mathbf{u}}|$ by some constants. It is also clear that there exist a constant $c_1$ such that
\begin{equation*}
\left| \frac{(\textbf{m}- \varrho\bar{\mathbf{u}}) \otimes (\varrho \bar{\mathbf{u}}-\textbf{m})}{\varrho} \right| \leq \frac{c_1}{2\varrho} |\textbf{m}-\varrho \bar{\mathbf{u}}|^2,
\end{equation*}
and a constant $c_2$ such that
\begin{equation*}
|p(\varrho) -p'(\bar{\varrho})(\varrho - \bar{\varrho})- p(\bar{\varrho})| \leq c_2 (P(\varrho) -P'(\bar{\varrho})(\varrho - \bar{\varrho})- P(\bar{\varrho})).
\end{equation*}
To deal with the term coming from It\"{o} correction terms, we follow \cite{BrFeHo2015A} and rewrite
\begin{equation}
\begin{aligned}
\label{nio0}
&\varrho\bigg\vert \frac{\mathbf{G}_k(\varrho,\varrho \mathbf{u})}{\varrho}  -\frac{\mathbf{G}_k(\bar{\varrho}, \bar{\varrho}\bar{\mathbf{u}})}{\bar{\varrho}}  \bigg\vert^2
= \chi_{\{\varrho \leq \bar{\varrho}/2\}}\, \varrho\bigg\vert \frac{\mathbf{G}_k(\varrho,\varrho \mathbf{u})}{\varrho}  -\frac{\mathbf{G}_k(\bar{\varrho}, \bar{\varrho} \bar{\mathbf{u}})}{\bar{\varrho}}  \bigg\vert^2
\\
& \qquad + \chi_{\{ \bar{\varrho}/2 < \varrho < 2\bar{\varrho}\}}\, \varrho\bigg\vert \frac{\mathbf{G}_k(\varrho,\varrho \mathbf{u})}{\varrho}  -\frac{\mathbf{G}_k(\bar{\varrho}, \bar{\varrho} \bar{\mathbf{u}})}{\bar{\varrho}}  \bigg\vert^2 
+ \chi_{\{  \varrho \geq 2\bar{\varrho}\}}\, \varrho\bigg\vert \frac{\mathbf{G}_k(\varrho,\varrho \mathbf{u})}{\varrho}  -\frac{\mathbf{G}_k(\bar{\varrho}, \bar{\varrho} \bar{\mathbf{u}})}{\bar{\varrho}}  \bigg\vert^2
\\
&=: I_1+I_2+I_3
\end{aligned}
\end{equation}
We can now use the inequality $\varrho\leq 1+ \varrho^\gamma$ and elementary inequalities to conclude that
\begin{equation*}
\begin{aligned}
\label{nio1}
I_1 &\leq 
c\, \chi_{\{\varrho \leq \bar{\varrho}/2\}}\,\bigg( \frac{1}{\varrho}\,\vert \mathbf{G}_k(\varrho ,\varrho {\mathbf{u}}) \vert^2  +\frac{\varrho}{\bar{\varrho}^2}\,\vert \mathbf{G}_k(\bar{\varrho},\bar{\varrho}\bar{\mathbf{u}}) \vert^2  \bigg) 
\\
&\leq c\, \chi_{\{\varrho \leq  \bar{\varrho}/2\}}\big( \varrho + \varrho\vert {\mathbf{u}}\vert^2  + \varrho\vert  \bar{\mathbf{u}}\vert^2\big) 
\leq c(R)\,\chi_{\{\varrho \leq  \bar{\varrho}/2\}}\big(1+ \varrho^\gamma +\varrho \vert {\mathbf{u}} -  \bar{\mathbf{u}}\vert^2\big) 
\end{aligned}
\end{equation*}
Therefore,
\begin{equation*}
\begin{aligned}
\sum_{k\in\mathbb{N}} \int_{\mathbb{T}^3} \big \langle {\mathcal{V}^{\omega}_{t,x}}; I_1\big \rangle \,\mathrm{d}x 
\leq c(R)\,
\mathcal{E}^1_{\mathrm{mv}} \big(\varrho ,{\mathbf{m}}\left\vert \right. \bar{\varrho}, \bar{\mathbf{u}}  \big), \,\, \text{for a.e.}\,\, t \in [0,T].
\end{aligned}
\end{equation*}
We can estimate $I_2$ and $I_3$ in a similar fashion.
Finally, we can conclude from \eqref{nio0} that
\begin{equation}
\begin{aligned}
\label{r3est}
\frac{1}{2}
\sum_{k\in\mathbb{N}}
\int_{\mathbb{T}^3} \Bigg \langle {\mathcal{V}^{\omega}_{t,x}}; \varrho\bigg\vert \frac{\mathbf{G}_k(\varrho,\varrho \mathbf{u})}{\varrho}  -\frac{\mathbf{G}_k(\bar{\varrho}, \bar{\varrho} \bar{\mathbf{u}})}{\bar{\varrho}}  \bigg\vert^2 \Bigg \rangle \,\mathrm{d}x 
\leq c(R)\,
\mathcal{E}^1_{\mathrm{mv}} \big(\varrho ,{\mathbf{m}}\left\vert \right. \bar{\varrho}, \bar{\mathbf{u}}  \big), \,\, \text{for a.e.}\,\, t \in [0,T].
\end{aligned}
\end{equation}
Collecting all the above estimates, we have shown that
\begin{equation}
\begin{aligned}
\label{r4est}
\int_0^{t \wedge \mathfrak{t}_R}
\mathcal{R}_{\mathrm{mv}} \big(\varrho,{\bf m} \left\vert \right. \bar{\varrho}, \bar{\mathbf{u}}  \big)  \,\mathrm{d}s
\leq c(R)\,\int_0^{t \wedge \mathfrak{t}_R} \mathcal{E}_{\mathrm{mv}} \big(\varrho,{\bf m} \left\vert \right. \bar{\varrho}, \bar{\mathbf{u}}  \big) 
(s) \,\mathrm{d}s .
\end{aligned}
\end{equation}
Combining \eqref{r4est} and \eqref{relativeEntropy1} and applying Gronwall's lemma yields
\begin{align}
\nonumber
\mathbb{E}\,  \Big[\mathcal{E}_{\mathrm{mv}} \big(\varrho ,{\mathbf{m}}\left\vert \right. \bar{\varrho}, \bar{\mathbf{u}}  \big)  
(t \wedge \mathfrak{t}_R) \Big] 
\leq c(R)\,
\mathbb{E}\,\Big[\mathcal{E}_{\mathrm{mv}} \big(\varrho ,{\mathbf{m}}\left\vert \right. \bar{\varrho}, \bar{\mathbf{u}}  \big)(0)\Big].
\label{relativeEntropy2}
\end{align}
Note that we have
\begin{equation*}
\begin{aligned}
\mathcal{E}_{\mathrm{mv}} \big(\varrho ,{\mathbf{m}}\left\vert \right. \bar{\varrho}, \bar{\mathbf{u}}  \big)  (0) 
&=
\int_{\mathbb{T}^3} \Big \langle {\mathcal{V}^{\omega}_{0,x}}; \frac{1}{2}\varrho_{0}\big\vert \mathbf{u}_{0} - \bar{\mathbf{u}}_{0} \big\vert^2  + P\big(\varrho_{0} ,\bar{\varrho}_0 \big) \Big \rangle \,\mathrm{d}x
\end{aligned}
\end{equation*}
which is zero in expectation by assumptions.
Therefore, we conclude that 
\begin{align*}
\mathbb{E}\,  \Big[\mathcal{E}_{\mathrm{mv}} \big(\varrho ,{\mathbf{m}}\left\vert \right. \bar{\varrho}, \bar{\mathbf{u}}  \big)  
(t \wedge \mathfrak{t}_R) \Big] =0, \quad \text{for all}\,\, t \in [0,T].
\end{align*}
This, in particular, implies that 
\begin{align*}
\lim_{\tau \rightarrow 0+} \frac{1}{\tau} \int_t^{t +\tau} \mathbb{E}\,  \Big[\mathcal{E}_{\mathrm{mv}} \big(\varrho ,{\mathbf{m}}\left\vert \right. \bar{\varrho}, \bar{\mathbf{u}}  \big)  
(s \wedge \mathfrak{t}_R) \Big]\,ds =0
\end{align*}
Thanks to \emph{a priori} estimates \eqref{apv}--\eqref{est:rhobfu22} (which are preserved in the limit), an application of Fubini's theroem, and a Lebesgue point argument reveals that for a.e. $t \in [0,T]$
\begin{align*}
\mathbb{E}\,  \Big[\mathcal{E}^1_{\mathrm{mv}} \big(\varrho ,{\mathbf{m}}\left\vert \right. \bar{\varrho}, \bar{\mathbf{u}}  \big)  
(t \wedge \mathfrak{t}_R) \Big] =0.
\end{align*}
In other words, for a.e. $(t,x) \in [0,T]\times \T^3$
\begin{equation*}
\mathcal{D}(t \wedge \mathfrak{t})=0, \quad \mathcal{V}^{\omega}_{t \wedge \mathfrak{t},x}	= \delta_{\bar{\varrho}(t \wedge \mathfrak{t},x), (\bar{\varrho}\bar{\bf u})(t \wedge \mathfrak{t},x)}, \,\p-\mbox{a.s.}
\end{equation*}

\section{Singular Limits}
\label{singular}
In this section, we discuss another application of relative energy - a rigorous justification of \emph {low Mach number limit} (also called \emph{incompressible limit}) for the system (\ref{P1})--(\ref{P2}). 
For this purpose, let us first rescale the deterministic counterpart of the stochastic compressible Euler system (\ref{P1})--(\ref{P2}) by non-dimensionalization. After combining terms appropriately (setting the so--called Strouhal number equal to one), and adding a stochastic force term, one reaches the following system 
\begin{align} \label{compEulerEp1}
\D \vre + \Div (\vre \vue) \,dt &=0,\\ \label{compEulerEp2}
\D (\vre \vue) + \big[\Div (\vre \vue \otimes \vue) + \frac{1}{\ep^2} \Grad p(\vre)\big] \dt&= \mathbb{G} (\vre, \vre \vue)\, \D W,
\end{align}
where $\ep$ is called the Mach number. It represents the norm of the velocity divided by the sound speed. It is well known that acoustic waves are responsible for the weak convergence of gradient part of the velocity. Hence, the limit of stochastic forcing term $\mathbb{G} (\vr, \vr \vu)\, \D W$ can be performed only if $\mathbb{G}$ is linear in the second variable. Therefore, for the remainder of the paper, we only consider $\mathbb{G}$ of the following form:
\begin{align} \label{assump}
\tn{G}(\vr, \vr \vu) = \vr \tn{K} + \vr \vu \tn{L},\quad \sum_{k \geq 1} \big( |K_k| + |L_k| \big) < \infty,
\end{align}
where $K_k, L_k$ are real numbers, and $\tn{K} = (K_k)_{k\in \N}$, $\tn{L} = (L_k)_{k\in \N}$ are suitable Hilbert-Schmidt operators.

We consider the {asymptotic limit} of solutions $(\vre, \vue)$ for $\ep \rightarrow 0$. Accordingly, as $\ep \rightarrow 0$, the speed of the acoustic wave becomes infinite and the fluid density approaches to a constant and the velocity becomes solenoidal. The resulting limiting equations are
\begin{align}
\label{incompEuler1}
\Div \vc{v} &= 0 \\ \label{incompEuler2}
\D \vc{v} + \big[\vc{v} \cdot \Grad \vc{v} + \Grad \Pi \big] \,dt &=  \mathbb{G} (1, \vc{v})\, \D W
\end{align}
There are couple of approaches available in literature to deal with the singular limit problem. The first approach deals with the classical (strong) solution of \eqref{P1}--\eqref{P2}, while the second approach is based on the concept of weak (dissipative measure-valued)
solutions for the system (\ref{P1}--\ref{P2}). For the deterministic counterpart of \eqref{P1}--\eqref{P2}, Kleinermann and Majda \cite{KM1}, Schochet \cite{SCHO2}, Masmoudi \cite{Masmoudi} and many others have successfully implemented the first approach. On the other hand, Feireisl et.al. \cite{Fei02}, and Bruell and Feireisl \cite{Bru} have explored the second approcah for deterministic compressible fluid equations. The main advantage of the second approach is that measure-valued solutions exist globally in time, while classical solutions may not exist globally in time.

Our aim is to extend the result of Feireisl et. al. \cite{Fei02} to the stochastically driven compressible fluids. To fix the ideas, let $\big[ \big(\Omega,\mathfrak{F}, (\mathfrak{F}_{{\varepsilon},t})_{t\geq0},\mathbb{P} \big); {\mathcal{V}^{\omega,\ep}_{t,x}}, W \big]$ be dissipative measure-valued martingale solutions to the system \eqref{compEulerEp1}--\eqref{compEulerEp2}, in the sense of Definition~\ref{def:dissMartin}. A rigorous justification of passing to the limit as $\ep \rightarrow 0$ in \eqref{compEulerEp1}--\eqref{compEulerEp2} makes use of the relative energy inequality \eqref{relativeEntropy}. However, since the filtration corresponding to measure-valued martingale solution depends explicitly on $\ep$, we need to justify that we can apply relative energy inequality to the pair ${\mathcal{V}^{\omega,\ep}_{t,x}}$ and $v$, for all $\ep>0$. Indeed, this can be done thanks to the fact that the unique strong solution $v$ can be constructed on any given stochastic basis and is adapted to the Brownian motion which is indepependent of $\ep$. In other words, for every filtration $(\mathfrak{F}_{{\varepsilon},t})_{t\geq0}$ the strong solution to the incompressible Euler equation is adapted to that filtration, since it is adapted to the brownian motion and the brownian motion is adapted to $(\mathfrak{F}_{{\varepsilon},t})_{t\geq0}$.


\subsection{Solutions of the incompressible Euler system}

Let us assume that we are given the stochastic basis $\StoB$ and the Wiener process $W$ identified in the beginning of this manuscript.

\begin{Definition}\label{def:strsolE}

Let $\StoB$ be a stochastic basis with a complete right-continuous filtration, let $W$ be an $( \mathfrak{F}_t )_{t \geq 0} $-cylindrical Wiener process. A
stochastic process $\vc{v}$ with a stopping time $\mathfrak{t}$ is called a (local) strong solution
to the Euler system (\ref{incompEuler1}), (\ref{incompEuler2}) provided
\begin{itemize}
\item the velocity $\vc{v} \in C([0,T]; W^{3,2}(\tor; \mathbb{R}^3))$, $\mathbb{P}$-a.s. is $( \mathfrak{F}_t )_{t \geq 0}$-adapted,
\[
\E\bigg[ \sup_{t \in [0,T]} \| \vc{v} (t, \cdot) \|_{W^{3,2} (\tor; \mathbb{R}^3 )}^p \bigg] < \infty, \,\, \mbox{for all}\,\, 1 \leq p < \infty;
\]
\item There holds $\p$-a.s.
\begin{equation} \label{Form}
\begin{split}
\Div \vc{v} &= 0, \\
\vc{v} (t \wedge \mathfrak{t})  &= \vc{v} (0) - \int_0^{t \wedge \mathfrak{t}} {\vc{P}_H} \left[ \vc{v} \cdot \Grad \vc{v} \right] \dt  +
\int_0^{t \wedge \mathfrak{t}} \vc{P}_H \left[ {\tn{G}}(1,\vc{v} ) \right] \, \Dif W,
\end{split}
\end{equation}
 a.e. in $(0,T)\times\tor$.
Here $\vc{P}_H$ denotes the standard Helmholtz projection onto the space of solenoidal functions.
\end{itemize}
\end{Definition}

Regarding the local-in-time existence of strong solutions to the stochastic Euler system, under certain restrictions imposed on the forcing coefficients $\tn{G}$, we refer to the recent work by Glatt-Holtz and Vicol \cite[Theorem 4.3]{GHVic}. 
Here, as mentioned before, we assume a very simple form of $\tn{G}$, namely 
as in \eqref{assump}.
%
One of the advantage of such a choice is that the pressure $\Pi$ can be computed explicitly from \eqref{Form}.
Indeed noting that
\[
\vc{P}_H \left[ {\tn{G}}(1,\vc{v} ) \right] =  {\tn{G}}(1,\vc{v} ),
\]
we get
\begin{equation} \label{presE}
\Grad \Pi = - \vc{P}^\perp_H [\vc{v} \cdot \Grad \vc{v}] = - \Grad \Delta^{-1} \Div (\vc{v} \otimes \vc{v}).
\end{equation}
Accordingly, the second equation in (\ref{Form}) reads
\begin{equation} \label{momEu}
\vc{v} (t \wedge \mathfrak{t})  = \vc{v} (0) - \int_0^{t \wedge \mathfrak{t}}  \left[ \vc{v} \cdot \Grad \vc{v} \right] \dt -
 \int_0^{t \wedge \mathfrak{t}} \Grad \Pi \ \dt +
\int_0^{t \wedge \mathfrak{t}}  {\tn{G}}(1,\vc{v} )  \,\Dif W.
\end{equation}


%
%
%
%
%

\subsection{Main result}
We now state the main result of this section related to the rescaled stochastic compressible Euler system \eqref{compEulerEp1}--\eqref{compEulerEp2}.

\begin{Theorem} \label{thm:sing}
Let $\tn{G}$ be given as in \eqref{assump}, and the initial data $\vr_{0,\ep}$, $(\vr \vu)_{0,\ep}$, and $\vc{v}_0$ be given such that $\p$-a.s.
\[
\Bigg \lbrace \vr_{0, \ep}, (\vr \vu)_{0,\ep}  \in L^\gamma(\tor) \times L^{\frac{2 \gamma}{\gamma + 1}}(\tor;
\mathbb{R}^3) \ \Big| \ \vr_{0,\ep} \geq \underline \vr > 0,\ \frac{|\vr_{0,\ep} - 1|}{\ep} \leq \delta(\ep),\
|(\vr \vu)_{0,\ep} - \vc{v}_0 |
\leq \delta (\ep) \Bigg \rbrace,
\]
where
\[
\delta (\ep) \to 0 \ \mbox{as}\ \ep \to 0,
\]
and where
$\vc{v}_0$ is an $\mathfrak{F}_0$-measurable random variable,
\begin{align*}
&\vc{v}_0 \in W^{3,2}(\tor;R^3), \ \Div \vc{v}_0 = 0, \quad \text{$\mathbb{P}$-a.s.},\\
&\expe{ \| \vc{v}_0 \|_{W^{3,2}(\tor; R^3)}^p } < \infty, \,\, \mbox{for all}\,\, 1 \leq p < \infty.
\end{align*}
Let $\big[ \big(\Omega,\mathfrak{F}, (\mathfrak{F}_{\ep,t})_{t\geq0},\mathbb{P} \big); {\mathcal{V}^{\omega,\ep}_{t,x}}, W \big]$ be a dissipative measure-valued martingale solution to the system \eqref{compEulerEp1}--\eqref{compEulerEp2}, satisfying the compatibility condition (\ref{fourth condition measure-valued solutions}), and suppose $(\vc{v}, \mathfrak{t})$ defined on the same probability space $\big(\Omega,\mathfrak{F}, (\mathfrak{F}_{\ep,t})_{t\geq0},\mathbb{P} \big)$ is a unique local strong solution of the Euler system 
\eqref{incompEuler1}--\eqref{incompEuler2} driven by the same cylindrical Wiener process $W$ in the sense of Definition~\ref{def:strsolE}. 
Then as $\ep \to 0$, $\p$-a.s.
\[
\mathcal{D}^\ep \to 0 \ \mbox{in}\ L^\infty(0,T),
\]
\[
\esssup_{t \in (0,T)} \E \Bigg[\intON{
\left< {\mathcal{V}^{\omega,\ep}_{t,x}}; \frac{1}{2} \vr  \left| \frac{\vc{m}}{\vr} - \vc{v}(t,x) \right|^2 + \frac{1}{\ep^2} \Big( P(\vr) - P'(1)(\vr - 1) - P(1) \Big) \right>}\,(t \wedge \tau) \Bigg] \to 0
\]
\end{Theorem}

We remark that the above theorem asserts that the probability measures ${\mathcal{V}^{\omega}_{t,x}}$ shrink to their expected value as $\ep \to 0$, where the latter are characterized
by the constant value $1$ for the density and the solution $\vc{v}$ of the incompressible system. The situation considered in the above theorem corresponds to the so called well-prepared data. However, one can extend these results to the case of ill-prepared data (refer to Masmoudi \cite{Masmoudi}, for the related deterministic results) with additional technical computations.

\subsection{Proof of Theorem~\ref{thm:sing}}
\label{I}
Let us first assume that $\vc{v}$, with a stopping time $\mathfrak{t}$, is a local strong solution of the stochastic Euler system \eqref{Form}. For each $M>0$, let us define
$$
\tau_M := \inf \bigg \lbrace t \in [0,T]: \| \nabla_x v(t)\|_{L^{\infty}(\T^3)} >M \bigg \rbrace,
$$
be another stopping time. Thanks to the existence theorem \cite[Theorem 4.3]{GHVic}, we assume, without loss of generality, that $\tau_M \le \mathfrak{t}$. Following \eqref{rell}, 
for ${\mathcal{V}^{\omega,\ep}_{t,x}}$ - the dissipative measure-valued solution of the rescaled system - we denote for a.e. $t \in [0,T]$
\[
\mathcal{E}_{\mathrm{mv}}^{\ep,1} \left( \vr, \vc{m} \Big| 1, \vc{v} \right)(t)
= \intON{
\left< {\mathcal{V}^{\omega,\ep}_{t,x}}; \frac{1}{2} \vr  \left| \frac{\vc{m}}{\vr} - \vc{v}(t,x) \right|^2 + \frac{1}{\ep^2} \Big( P(\vr) - P'(1)(\vr - 1) - P(1) \Big) \right>} + \mathcal{D}^{\varepsilon}(t),
\]
the relative energy functional associated to $1$, $\vc{v}$. Similarly, following \eqref{rell1}
and \eqref{rell2}, we define $\mathcal{E}_{\mathrm{mv}}^{\ep,2} \left( \vr, \vc{m} \Big| 1, \vc{v} \right)$, and $\mathcal{E}_{\mathrm{mv}}^{\ep} \left( \vr, \vc{m} \Big| 1, \vc{v} \right)$ appropriately.

\subsubsection{Relative energy inequality}
As the quantities $r = 1$, $\vc{U}(\tau) = \vc{v}(\tau \wedge \tau_M)$ enjoy the required regularity, they can be used as test functions in the relative entropy inequality. Therefore, we conclude
\begin{equation} 
\label{I7}
\begin{aligned}
\mathcal{E}_{\mathrm{mv}}^\ep &\left(\vr, \vc{m} \ \Big| \ 1, \vc{v} \right)(\tau \wedge \tau_M) \leq
\intON{
\left< {\mathcal{V}^{\omega,\ep}_{0,x}}; \frac{1}{2} \vr  \left| \frac{\vc{m}}{\vr} - \vc{v}_0(x) \right|^2 + \frac{1}{\ep^2} \Big( P(\vr) - P'(1)(\vr - 1) - P(1) \Big) \right>} \\
&+
\int_0^{\tau \wedge \tau_M} \intON{ \left[ \left< {\mathcal{V}^{\omega,\ep}_{t,x}}; \vr \vc{v}(t,x) - \vc{m} \right> \cdot D^d_t \vc{v} +
\left< {\mathcal{V}^{\omega,\ep}_{t,x}}; (\vr \vc{v}(t,x) - \vc{m} ) \otimes \frac{\vc{m}}{\vr} \right> : \Grad \vc{v}  \right] } \dt \\
&- \int_0^{\tau \wedge \tau_M} \int_{\mathbb{T}^3} \Grad \vc{v}:{\rm d}\mu^{\ep}_m  +M_{RE}(\tau \wedge \tau_M) - M_{RE}(0)  \\
&+ \frac{1}{2}\sum_{k\geq1}\int_0^{\tau \wedge \tau_M} \intTor{ \Bigg \langle {\mathcal{V}^{\omega,\ep}_{t,x}}; \varrho \Big|\frac{1}{{\varrho}}  {\vc{G}_k}(\varrho,\varrho \vu) - \vc{G}_k(1, \vc{v}) \Big|^2 \Bigg \rangle}\ \dt + \frac12  \int_0^{\tau \wedge \tau_M}\int_{\T^3} d \mu^{\ep}_e.
\end{aligned}
\end{equation}
As the initial data are well-prepared, we get
\begin{equation} \label{I7a1}
\E \bigg[\intON{
\left< {\mathcal{V}^{\omega,\ep}_{0,x}}; \frac{1}{2} \vr  \left| \frac{\vc{m}}{\vr} - \vc{v}_0(x) \right|^2 + \frac{1}{\ep^2} \Big( P(\vr) - P'(1)(\vr - 1) - P(1) \Big) \right>} \bigg] \le \delta(\ep),\, \delta(\ep) \to 0 \ \mbox{as}\ \ep \to 0.
\end{equation}
In addition, since the compatibility condition (\ref{fourth condition measure-valued solutions}) is satisfied for all $\ep$, we deduce
\begin{equation} \label{I7a2}
\frac12 \int_0^{\tau \wedge \tau_M} \int_{\T^3} d \mu^{\ep}_e - \int_0^{\tau \wedge \tau_M} \int_{\mathbb{T}^3} \Grad \vc{v}:{\rm d}\mu^{\ep}_m 
\leq C M \int_0^{\tau \wedge \tau_M} \mathcal{D}^\ep \ \dt
\end{equation}
Next, motivated by the specific form of $\tn{G}(1, \vc{v})$ introduced in (\ref{assump}), and thanks to the boundedness of the sum $\sum_{k\geq1}|L_k|^2$,
\begin{equation}
\label{I7a22}
\begin{aligned}
\sum_{k\geq1}& \intTor{ \Bigg \langle {\mathcal{V}^{\omega,\ep}_{t,x}};  \varrho \Big|\frac{1}{{\varrho}}  {\vc{G}_k}(\varrho,\varrho \vu) - \vc{G}_k(1, \vc{v}) \Big|^2 \Bigg \rangle}  \\
& \qquad \qquad =\sum_{k\geq1}
\intTor{ \Bigg \langle {\mathcal{V}^{\omega,\ep}_{t,x}};  \varrho \left| (\vu - \vc{v}) H_k \right|^2 \Bigg \rangle} \leq \,C\, \mathcal{E}^{\ep}_{\mathrm{mv}} \left( \vr, \vm \ \Big| \ 1 , \vc{v} \right).
\end{aligned}
\end{equation}
We proceed further, and claim that
\begin{equation} \label{I7a3}
\begin{split}
\int_0^{\tau \wedge \tau_M} & \intON{ \left[ \left< {\mathcal{V}^{\omega,\ep}_{t,x}}; \vr \vc{v}(t,x) - \vc{m} \right> \cdot D^d_t \vc{v} +
\left< {\mathcal{V}^{\omega,\ep}_{t,x}}; (\vr \vc{v}(t,x) - \vc{m} ) \otimes \frac{\vc{m}}{\vr} \right> : \Grad \vc{v}  \right] } \dt
\\ & \qquad \leq C M \int_0^{\tau \wedge \tau_M}  \mathcal{E}^{\ep}_{\mathrm{mv}} \left(\vr, \vc{m} \ \Big| \ 1, \vc{v} \right)\,\dt.
\end{split}
\end{equation}
In order to justify (\ref{I7a3}), we first start by writing
\[
\begin{split}
&\intON{ \left< {\mathcal{V}^{\omega,\ep}_{t,x}}; (\vr \vc{v}(t,x) - \vc{m} ) \otimes \frac{\vc{m}}{\vr} \right> : \Grad \vc{v} } \\&=
\intON{ \left< {\mathcal{V}^{\omega,\ep}_{t,x}}; (\vr \vc{v}(t,x) - \vc{m} ) \otimes \frac{\vc{m} - \vr \vc{v} }{\vr} \right> : \Grad \vc{v} }
+ \intON{ \left< {\mathcal{V}^{\omega,\ep}_{t,x}}; \vr \vc{v}(t,x) - \vc{m} \right> \cdot \vc{v} \cdot \Grad \vc{v} },
\end{split}
\]
where obviously,
\[
\intON{ \left< {\mathcal{V}^{\omega,\ep}_{t,x}}; (\vr \vc{v}(t,x) - \vc{m} ) \otimes \frac{\vc{m} - \vr \vc{v} }{\vr} \right> : \Grad \vc{v} }
\leq C M\, \mathcal{E}^{\ep}_{\mathrm{mv}} \left(\vr, \vc{m} \ \Big| \ 1, \vc{v} \right).
\]
Moreover, as $\vc{v}$ fulfills equation (\ref{incompEuler2}), we observe that justifying (\ref{I7a3}) reduces to showing
\begin{equation} \label{I7a4}
\int_0^{\tau \wedge \tau_M}  \intON{ \left< {\mathcal{V}^{\omega,\ep}_{t,x}}; \vc{m} - \vr \vc{v}(t,x) \right> \cdot \Grad \Pi } \dt
\leq C \int_0^{\tau \wedge \tau_M}  \mathcal{E}^{\ep}_{\mathrm{mv}} \left(\vr, \vc{m} \ \Big| \ 1, \vc{v} \right)\,\dt.
\end{equation}
To see (\ref{I7a4}), we deduce from the density equation that
\begin{equation} \label{I9a}
\begin{split}
\int_0^{\tau \wedge \tau_M} &\intON{ \left< {\mathcal{V}^{\omega,\ep}_{t,x}}; \vc{m} \right> \cdot \Grad \Pi }\dt \\ & \qquad = \left[ \intON{ \left< {\mathcal{V}^{\omega,\ep}_{t,x}}; \vr \right> \Pi } \right]_{t = 0}^{t = \tau  \wedge \tau_M} { - \int_0^{\tau \wedge \tau_M} \int_{\mathcal{T}^N} \Grad \Pi \cdot {\rm d} \mu^{\ep}_c } \\& \qquad \qquad =
\ep \left[ \intON{ \left< {\mathcal{V}^{\omega,\ep}_{t,x}}; \frac{\vr - 1}{\ep} \right> \Pi } \right]_{t = 0}^{t = \tau  \wedge \tau_M} - \int_0^{\tau \wedge \tau_M} \int_{\mathcal{T}^N} \Grad \Pi \cdot {\rm d} \mu^{\ep}_c .
\end{split}
\end{equation}
Similarly, we may use the incompressibility condition $\Div \vc{v} = 0$ to obtain
\begin{equation} \label{I9b}
\int_0^{\tau \wedge \tau_M} \intON{ \left< {\mathcal{V}^{\omega,\ep}_{t,x}}; \vr \vc{v}(t,x) \right> \cdot \Grad \Pi } \dt
= \ep \int_0^{\tau \wedge \tau_M} \intON{ \left< {\mathcal{V}^{\omega,\ep}_{t,x}}; \frac{\vr - 1}{\ep}  \right> \vc{v} \cdot \Grad \Pi } \dt
\end{equation}
Note that $\Pi$ and $\Grad \Pi$ are bounded continuous in
$[0,T] \times \mathbb{T}^3$. Now observe that the right-most integral in (\ref{I9a}) can be controlled by the dissipation defect $\mathcal{D}^\ep$. In order to control the other term on the right side of (\ref{I9a}) and the term on the right side of \eqref{I9b}, we evoke again relative energy inequality, this time for $r = 1$, $\vc{U} = 0$ obtaining
\[
\begin{split}
&\expe{ \intTor{ \Bigg \langle {\mathcal{V}^{\omega,\ep}_{t,x}}; \left[ \frac{1}{2} \vr |\vu|^2 + \frac{1}{\ep^2} \left( P(\vr) - P'(1)(\vr - 1) - P(1) \right) \right] \Bigg \rangle}(\tau \wedge \tau_M) + \mathcal{D}^{\varepsilon}(\tau \wedge \tau_M) } \\
&\qquad \leq
\expe{ \intTor{ \Bigg \langle {\mathcal{V}^{\omega,\ep}_{0,x}}; \left[ \frac{1}{2} \vr |\vu|^2 +  \frac{1}{\ep^2} \left( P(\vr) - P'(1)(\vr - 1) - P(1) \right) \right] \Bigg \rangle}(0) }.
\end{split}
\]
Thus, if the right-hand side of the above inequality is bounded uniformly for $\ep \to 0$, we deduce the following uniform bounds and set $\gamma_*=\min\lbrace \gamma, 2\rbrace$:
\begin{align*}
\E \Big[\int_{\T^3} \left< {\mathcal{V}^{\omega,\ep}_{t,x}}; \frac12 \vr |\vu|^2 \right> (\tau \wedge \tau_M)\,dx\ \Big] \le C, \quad
\E \Big[\int_{\T^3} \left< {\mathcal{V}^{\omega,\ep}_{t,x}}; \frac{|\vr -1|^{\gamma_*}}{\ep^2}\right> (\tau \wedge \tau_M)\,dx\ \Big] \le C
\end{align*}
Now using (\ref{presE}), the continuity of $\nabla \Delta^{-1} \mathrm{div}$, and regularity of $\vc{v}$, we obtain 
\begin{align*}
&\left| \expe{ \int_0^{\tau \wedge \tau_M} \intTor{ \left< {\mathcal{V}^{\omega,\ep}_{t,x}}; \frac{ \varrho - 1}{\ep} \right>\,\Grad \Pi \cdot \vc{v} } \ \dt } \right| \\
& \qquad \le \E\Bigg[\Big\| \left< {\mathcal{V}^{\omega,\ep}_{t,x}}; \frac{ \varrho - 1}{\ep} \right>\Big\|_{L^{\gamma_*}(\T^3)}\Bigg] \,\E\Big[\|\Grad \Pi \|_{L^{2\gamma^{'}_*}(\T^3)}\Big]\, \E\Big[\|\vc{v} \|_{L^{2\gamma^{'}_*}(\T^3)}\Big] \\
& \qquad \qquad \le \E\Bigg[ \Big\| \left< {\mathcal{V}^{\omega,\ep}_{t,x}}; \frac{ \varrho - 1}{\ep} \right>\Big\|_{L^{\gamma_*}(\T^3)} \Bigg] \,\E\Big[\| \vc{v} \otimes \vc{v} \|_{L^{2\gamma^{'}_*}(\T^3)} \Big] \,\E \Big[\|\vc{v} \|_{L^{2\gamma^{'}_*}(\T^3)}\Big] \le C,
\end{align*}
uniformly for $\ep \to 0$. In particular, the last two relevant terms on the right-hand side of (\ref{I9a}) and (\ref{I9b}) vanish for $\ep \to 0$. In view of this result, along with (\ref{I7a1}), (\ref{I7a2}), and (\ref{I7a22}), the proof of the theorem follows after taking expectation in \eqref{I7}, and applying Gronwall's lemma.

\section*{Acknowledgements}
U.K. acknowledges the support of the Department of Atomic Energy,  Government of India, under project no.$12$-R$\&$D-TFR-$5.01$-$0520$, and India SERB Matrics grant MTR/$2017/000002$.

\end{document}